\documentclass[12pt,leqno]{amsart}
\usepackage{epsf, amsmath, amsfonts, amsthm, amssymb}
\newtheorem{thm}{Theorem}[section]
\newtheorem{lem}[thm]{Lemma}
\newtheorem{prop}[thm]{Proposition}
\newtheorem{cor}[thm]{Corollary}

\newtheorem{prob}[thm]{Problem}
\theoremstyle{definition}
\newtheorem{defin}[thm]{Definition}
\newtheorem*{remark}{Remark}

\newcommand{\N}{\mathbb N}
\newcommand{\R}{\mathbb R}
\newcommand{\C}{\mathbb C}
\newcommand{\Q}{\mathbb Q}
\newcommand {\A}{\mathbb A}

\newcommand{\cE}{\mathcal{E}}

\newcommand{\vp}{\varepsilon}
\newcommand{\vpb}{\overline{\varepsilon}}

\newcommand{\xb}{\overline x}
\newcommand{\ib}{\overline\iota}
\newcommand{\rhob}{\overline \rho}
\newcommand{\zb}{\overline z}

\newcommand{\supp}{{\rm{supp}}}
\newcommand{\ran}{{\rm{ran}}}

\newcommand{\coo}{c_{00}}
\newcommand{\eqdist}{=_{\text{dist}}}

\def\hangbox to #1 #2{\vskip1pt\hangindent #1\noindent \hbox to
#1{#2}$\!\!$}

\newcommand{\tn}{|\!|\!|}

\newcommand{\Btn}{\Big|\!\Big|\!\Big|}
\newcommand{\ra}{\rangle}
\newcommand{\la}{\langle}
\newcommand{\nl}{\langle\hskip-3pt\langle}
\newcommand{\nr}{\rangle\hskip-3pt\rangle}
\newcommand{\Bnl}{\Big\langle\hskip-5pt\Big\langle}
\newcommand{\Bnr}{\Big\rangle\hskip-5pt\Big\rangle}

\newcommand{\kleq}{\!\leq\!}
\newcommand{\kge}{\!\ge\!}
\newcommand{\kle}{\!<\!}

\newcommand{\kin}{\!\in\!}

\newcommand{\st}{{\tilde s}}
\newcommand{\pt}{\tilde p}

\newcommand{\xt}{\tilde x}

\newcommand{\yt}{\tilde y}
\newcommand{\zt}{\tilde z}
\newcommand{\gt}{\tilde g}
\newcommand{\mt}{\tilde m}

\newcommand{\lt}{\tilde l}
\newcommand{\Et}{\tilde E}

\newcommand {\del}{ \; \big| \;}
\newcommand{\norm}[1]{\lVert#1\rVert}

\title{Subsequential minimality in  Gowers and Maurey spaces}
\author {Valentin Ferenczi and Thomas Schlumprecht}
\begin{document}
\thanks{The second author's research is partially supported by NSF grant DMS0856148.
A large part of the research for this paper was accomplished during a visit of the second author to the  Universidade de S\~ao Paulo which was funded by FAPESP grant 2010/17493-1.
}

\begin{abstract} We  define block sequences  $(x_n)$ in every block subspace of  a variant of the space of Gowers and Maurey
 so that the map $x_{2n-1}\mapsto x_{2n}  $ extends to an isomorphism. This implies the existence of a subsequentially minimal HI space, which solves a question in \cite{FR}. 

\end{abstract}

\maketitle

\tableofcontents

\section{Introduction}\label{intro}

We start this article by motivating our result with a presentation of W.T. Gowers's program
of classification of Banach spaces, and its recent developments along the lines of \cite{FR,FR2} and \cite{FG}.

\subsection{Gowers' classification program}

\

W.T. Gowers' fundamental results in geometry of Banach spaces \cite{g:hi,g:dicho} opened the way to a {\em loose classification of Banach spaces up to subspaces}, known as Gowers'  program. 
The aim of this program is to produce  a list of classes of infinite dimensional Banach spaces such that:

(a)  the classes are {\em hereditary}, i.e., stable under taking subspaces (or block subspaces),

(b) the classes are {\em inevitable}, i.e., every infinite dimensional Banach space contains a subspace in one of the classes,

(c) the classes are mutually disjoint,

(d) belonging to one class gives some information about the operators that may be defined on the space or on its subspaces.

\

We shall refer to such a list as a {\em list of  inevitable classes of
  Gowers}.
The reader interested in more details about  Gowers' program  may consult \cite{g:dicho} and \cite{FR}.
Let us just say that the class of spaces $c_0$ and $\ell_p$ is seen as the
  most regular class, and
  so, the objective this program really is the classification of those spaces which do not contain a copy of $c_0$ or $\ell_p$. 
We shall first  give a  summary of the classification obtained in \cite{FR}
and of the results of Gowers that led to it.

The first classification result of Gowers was motivated by his construction with B. Maurey of a hereditarily
indecomposable (or HI) space $GM$, i.e.,  a space such that no subspace may be
written as the direct sum of infinite dimensional subspaces \cite{GM}. The space $GM$ was the first known example of a space without an unconditional sequence.  Gowers then
proved his {\em first dichotomy}.

\begin{thm}[First dichotomy \cite{g:hi}]\label{T:0.1} Every Banach space contains either an HI subspace or  a subspace with an unconditional basis. \end{thm}

 These were the first two examples of
inevitable classes.

After  $GM$ was defined, Gowers   was  able to apply a criterion of  P.G. Casazza to prove that an unconditional Gowers-Maurey's space $G_u$ is  isomorphic to no proper subspace,  solving Banach's hyperplane problem \cite{g:hyperplanes}. Later on Gowers and Maurey proved that $GM$ also solves Banach's hyperplane problem, but as a consequence of general properties of HI spaces, based on Fredholm theory, rather than by applying the criterion.
Let us note in passing that our main result will  suggest that Casazza's criterion is indeed not satisfied in Gowers-Maurey's space.

  Gowers then 
refined the list by proving a {\em second dichotomy} as a consequence of his general Ramsey theorem for block sequences \cite{g:dicho}. A space is said to be {\em quasi-minimal} if any two subspaces  have further
subspaces which are isomorphic.

\begin{thm}[Second dichotomy \cite{g:dicho}] \label{T:0.2}Every Banach space contains a quasi-mini\-mal subspace or a subspace with a basis such that  no two disjointly supported block subspaces are
isomorphic. \end{thm}

Finally,
H. Rosenthal had defined  a space to be {\em minimal} if it embeds into any of its subspaces. A
quasi minimal space which does not contain a minimal subspace is called {\em
strictly quasi minimal}, so Gowers again divided the class of quasi minimal
spaces into the class of strictly quasi minimal spaces and the class of minimal
spaces.

\

Gowers deduced from these dichotomies a list of  four inevitable classes of Banach spaces:
HI spaces, such as $GM$; spaces with bases such that no disjointly supported
subspaces are isomorphic, such as $G_u$;
strictly quasi minimal spaces with an unconditional basis, such as Tsirelson's space  $T$ \cite{tsi} ; and
finally, minimal spaces, such as $c_0$ or $\ell_p$, but also $T^*$,
Schlumprecht's space $S$ \cite{Sch1}, or its dual $S^*$ \cite{MP}.

\

In \cite{FR} several other  dichotomies for Banach spaces  were
obtained. The first one, called 
 the {\em third dichotomy}, refines the distinction between the minimality of Rosenthal and strict quasi-minimality. Given a Banach space $X$ with a basis $(e_n)$, a space $Y$ is {\em
tight} in $X$  if there is a sequence of successive
subsets $I_0<I_1<I_2<\ldots$ of $\N$, such that the support on $(e_n)$ of any isomorphic copy of $Y$ intersects all but finitely many of the $I_j$'s. In other words, for any infinite subset
$J$ of $\N$,
$$Y \not\sqsubseteq [e_i: i\in\N\setminus \bigcup_{j \in J} I_j],$$
where $\sqsubseteq$ means "embeds into".

The space $X$ itself is {\em tight} if  all subspaces $Y$ of $X$ are tight in $X$.

\

As observed in \cite{FG}, the tightness of a space $Y$ in $X$ allows the following characterization:
$Y$ is tight in $X$  if and only if
$$\big\{u \in 2^{\omega}: Y \sqsubseteq [e_n : n \in u]\big\}$$
is a meager subset of the Cantor space $2^{\omega}$.
Here we identify the set ot subsets of $\omega$ with the Cantor space
$2^{\omega}$, equipped with its usual topology.

 After  observing  that the tightness property is hereditary and incompatible with minimality, the authors of \cite{FR} prove:

\begin{thm}[Third dichotomy \cite{FR}]\label{T:0.3} Every Banach space contains a minimal subspace or a tight subspace.\end{thm}

Special types of tightness may be defined
 according to the way the $I_n$'s may be chosen in function of $Y$. It is observed in \cite{FR} that the actual known examples of tight spaces satisfy  one of two stronger forms of tightness, called {\em by range}, and  {\em with constants}. Thus e.g. Gowers unconditional space $G_u$ is tight by range, and Tsirelson's space $T$ is tight with constants, see also \cite{FR2} for other examples.

We shall be mainly interested in tightness by range, which we define in the next subsection. We refer to the end of the paper
for definitions and comments about tightness with constants.

\subsection{Ranges and supports}

\

The following distinction is essential. If $X$ is a space with a basis $(e_i)_i$,
then the definition of the {\em support} ${\rm supp}\ x$ of a vector $x$ is well-known: it is the set
$\{i \in \N: x_i \neq 0\}$,
where $x=\sum_{i=0}^{\infty} x_i e_i$.
On the other hand the {\em range}, ${\rm ran} \ x$, of  $x$ is the smallest interval of integers
containing its support. 
So of course, having finite range and having finite support are the same, but the range is always an interval of integers, while the support may be an arbitrary subset of $\N$.

If $Y=[y_n, n \in \N]$ is a block subspace of $X$, then the support of $Y$ is $\cup_{n \in \N} {\rm supp}\ y_n$, and the range of $Y$ is
$\cup_{n \in \N}{\rm ran}\ y_n$. 

Let us now recall the criterion of Casazza, which appears in \cite{g:hyperplanes}. Two basic sequences $(x_n)_{n \in \N}$ and $(y_n)_{n \in \N}$  are said to be equivalent if the map $x_n \mapsto y_n$ extends to an isomorphism of $[x_n, n \in \N]$ onto $[y_n, n \in \N]$.

\begin{prop} \cite{C}
Let $X$ be a Banach space with a basis. Assume that for any block sequence $(x_n)$ in $X$, $(x_{2n})$ is not equivalent to $(x_{2n+1})$. Then $X$ is isomorphic to no proper subspace.
\end{prop}

The criterion of Casazza leads to studying the possible isomorphisms between disjointly supported or disjointly ranged subspaces. As proved in \cite{FR}, this turns out to have an essential connection with the notion of tightness.
In what follows we shall say that two spaces are {\em comparable} if one embeds into the other.

\

If no two disjointly supported block-subspaces are isomorphic, then equivalently no two such subspaces are comparable. This is also equivalent to saying that 
 for
every block   subspace $Y$, spanned by a block sequence $(y_n)$,  the sequence of successive subsets
$I_0<I_1<\ldots$ of $\N$ witnessing the tightness of $Y$ in $(e_n)$ may be
defined by $I_k={\rm supp}\ y_k$ for each $k$.
When this happens it is said that $X$ is {\em  tight by support} \cite{FR}. So Gowers' second dichotomy
may be interpreted as a dichotomy between a form of tightness and a form of minimality, and $G_u$ is tight by support.

If now for
every block subspace $Y=[y_n]$, the sequence of successive subsets
$I_0<I_1<\ldots$ of $\N$ witnessing the tightness of $Y$ in $(e_n)$ may be
defined by $I_k={\rm ran}\ y_k$ for each $k$, then $X$ is said to be {\em tight by range}. This is equivalent to no two
block subspaces with disjoint ranges being comparable, a property which is formally weaker than tightness by support.
Note that the criterion of Casazza applies to prove that a space which is tight by range cannot be isomorphic to its proper subspaces.

The distinction between range and support is relevant here. While it is easy to check that
a basis which is tight by support must be unconditional, it is proved in \cite{FR2} that HI spaces may be tight by range; this is the case of an asymptotically unconditional and HI  Gowers-Maurey's space $G$,  due to Gowers \cite{g:asymptotic}.

In \cite{FR} it was proved  that there  also exists a dichotomy relative to
tightness by range.  The authors define a space $X$ with a
basis $(x_n)$  to be {\em subsequentially minimal} if every subspace of  $X$
contains an isomorphic copy of a subsequence of $(x_n)$. 
Tsirelson's space $T$ is the classical example of subsequentially minimal, non-minimal space.

\begin{thm}[Fourth dichotomy \cite{FR}]\label{main2}
Any Banach space contains a
subspace with a basis which is either tight by range or subsequentially minimal.
\end{thm}

The second case in Theorem \ref{main2}
may be improved to the following hereditary property of a basis $(x_n)$, that is called {\em
sequential minimality}: $(x_n)$ is quasi minimal and  every block sequence of
$[x_n]$ has a subsequentially minimal block sequence.

\subsection{The list of 6 inevitable classes}

 The first four dichotomies and the interdependence of the properties involved can be
visualized in the following diagram.

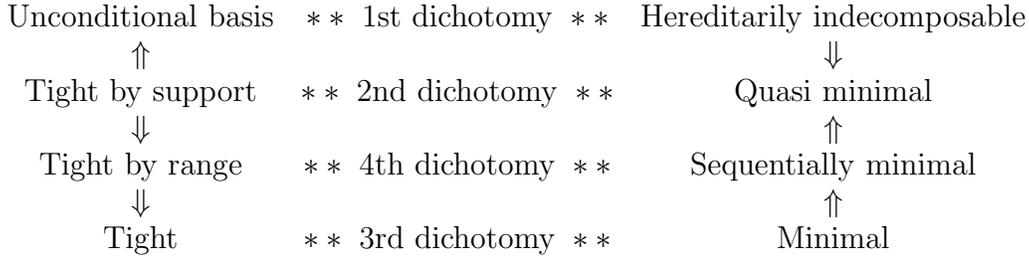
\[
\begin{tabular}{ccc}

Unconditional basis&$**\textrm{ 1st dichotomy }**$& Hereditarily indecomposable\\

$\Uparrow$&         &$\Downarrow$\\

Tight by support & $**\textrm{ 2nd dichotomy }**$ & Quasi minimal  \\

$\Downarrow$&&$\Uparrow$\\

Tight by range    &      $**\textrm{ 4th dichotomy }**$          & Sequentially minimal     \\

$\Downarrow$&&$\Uparrow$\\

Tight&                  $**\textrm{ 3rd dichotomy }**$              & Minimal\\

\end{tabular}
\]

\

The easy observation that HI spaces are  quasi-minimal is due to Gowers (see subsection \ref{SubS:1.5}).
On the other hand it was shown in \cite[Corollary 19]{GM}  and \cite[Theorem 21]{GM} that an HI space  cannot be isomorphic to any proper subspace.
This implies that an HI space cannot contain a minimal subspace.

Therefore by the third dichotomy, every HI space must contain a tight subspace, but it is unknown whether every HI space with a basis must itself be tight.
\

 Combining the four dichotomies and the relations between them, the following
list of 6 classes of Banach spaces contained in any Banach space is obtained
in \cite{FR}:

\begin{thm}\cite{FR}\label{final} Any infinite dimensional Banach space contains a subspace of one of
the types listed in the following chart:

\

\begin{center}
  \begin{tabular}{|l|l|l|}

    \hline

    Type       & Properties                              & Examples                                  \\

    \hline

    (1)                     & HI, tight by range    &   $G, G^*$ \\

    \hline

    (2) & HI, tight, sequentially minimal     &   $ {  \mathcal  X_{GM}}$ \\

    \hline

    (3) & tight by support  &   $G_u, G_u^*$, $X_u$,   $X_u^*$, $X_{abr}$  \\

    \hline

(4) & unconditional basis, tight by range, & \\
 & quasi minimal & ? \\

\hline

(5) & unconditional basis, tight,  & $T$, $T^{(p)}$ \\

    & sequentially minimal & \\

\hline

(6) & unconditional basis, minimal & $S,S^*$,  $T^*$, $c_0$, $\ell_p$\\

\hline
  \end{tabular}

\end{center}
\end{thm}

\

\

For information about the examples appearing in type (1) and (3)-(6) we refer to \cite{FR2}.
Two major open problems of \cite{FR} were whether spaces of type (2) or (4) existed.
The only known  proofs of sequential minimality used properties which implied unconditionality, so presumably the construction of a type (2) space would require new methods.

\

 The main result of this paper is the
existence of an example ${\mathcal X}_{GM}$ of type (2), similar to Gowers-Maurey's space, which is reported on the chart above.

\subsection{The main result}

\begin{thm}\label{maim} There exists a version ${\mathcal G \mathcal M}$ of Gowers-Maurey's space such that
\begin{itemize}
\item[(a)] ${\mathcal G \mathcal M}$ does not contain an unconditional basic sequence.
\item[(b)]  Any block subspace of ${\mathcal G \mathcal M}$ contains  a  block sequence
$(y_n)_n$  such that $(y_{2n})$ is equivalent to $(y_{2n+1})$.
\end{itemize}
\end{thm}

The proof of Theorem will be accomplished in Section \ref{S:6}, Theorem \ref{T:6.15}. 
The modification leading to ${\mathcal G \mathcal M}$ is essentially technical. 
Note that by (a) and the first dichotomy, ${\mathcal G \mathcal M}$ contains an HI subspace. 
So this subspace is not isomorphic to its proper subspaces, although by Theorem \ref{maim} (b), it does not satisfy
 Casazza's criterion.  
Using also the third and fourth dichotomy, we deduce that some subspace of ${\mathcal G \mathcal M}$ satisfies:

\begin{thm} There exists a tight, HI, sequentially minimal space ${\mathcal X}_{GM}$. \end{thm}

 It may be  surprising to see that the answer to the existence of type (2) spaces is given by a non-essential modification of
the first known example of HI space.
We actually believe that $GM$ itself satisfies Theorem \ref{maim} (b), and therefore fails to satisfy the criterion of Casazza.

\

We shall also observe that the space ${\mathcal G \mathcal M}$ is {\em locally minimal}, which means that all finite dimensional subspaces of ${\mathcal G \mathcal M}$ embed into all its infinite dimensional subspaces, with uniform constant.  Problem 5.2 from \cite{FR2} asked whether a sequentially and locally minimal should be minimal or at least contain a minimal subspace. We therefore answer this by the negative.

\begin{thm} There exists a locally minimal, sequentially minimal, tight space.
\end{thm} 

To conclude this subsection let us mention that our results hold both in the real and in the complex setting.

\subsection{Some comments on our construction}\label{SubS:1.5}

\

Let us first recall why HI spaces are quasi-minimal. Let $X$ be an HI space with a basis. If $\epsilon>0$, and  two block-subspaces $U$ and $V$ of $X$ are given, one can use the HI property to obtain two normalized block-sequences $(u_n)_n$ and $(v_n)_n$ in $U$ and $V$ respectively,
so that $\norm{u_n-v_n} \leq \epsilon 2^{-n}$ for all $n \in \N$.
So there is a compact perturbation of the canonical injection mapping $[u_n, n \in \N]$ onto $[v_n, n \in \N]$,  which are therefore isomorphic. 
Note that if $U$ and $V$ are disjointly supported, or even disjointly ranged, then each $u_n$ is disjointly {\em supported} from $v_n$. 

\

If now we want to obtain a canonical isomorphism between $[u_n, n \in \N]$ and $[v_n, n \in \N]$, so that $(u_n)$ and $(v_n)$ are disjointly {\em ranged} and seminormalized block sequences, then such a crude approach does not work. Let us explain this when $(u_n)$ and $(v_n)$ are {\em intertwined}, i.e. $u_0<v_0<u_1<v_1<u_2<\cdots$.  
By using the projection on the range of $u_n$, we see that the norm $\norm{u_n-v_n}$ is bounded below by a constant depending on the constant of the basis, and so the map $u_n \mapsto u_n-v_n$ can never be compact.
We may however hope to pick $u_n$ and $v_n$ so that this map is strictly singular.
Actually in the case when $X$ is, say, complex HI, we  must do so. Indeed, we  know in this case \cite{Fe}  that there must exist $\lambda \in \C$ and  
 a strictly singular operator $S:[u_j:j\kin\N]\to X$,
   such that $v_n-\lambda u_n= S(u_n)$; so by projecting on ${\rm ran}\ v_n$ we get that $S(u_n)$ is bounded below, and that $S$ is  strictly singular non compact from $[u_n, n \in \N]$ into $X$.

So our result of existence of two intertwined and equivalent block sequences in any subspace of ${\mathcal G \mathcal M}$ will be related to the techniques of the construction of strictly singular non-compact operators on subspaces of Schlumprecht's space $S$ and of $GM$ type spaces, as appears   in \cite{AS} and \cite{Sch2}.
We shall replace the condition that $\|u_n-v_n\| \leq \epsilon 2^{-n}$  by the requirement that the sequence $u_n-v_n$ generates a spreading model which is "largely"  dominated by the spreading models of $u_n$, $v_n$ and $u_n+v_n$. 
From some techniques of \cite{Sch2}, this will imply that the map taking $u_n$ to $u_n-v_n$ extends to a bounded (actually strictly singular) map, and the same for the map taking $v_n$ to $u_n-v_n$. Therefore $(u_n)$ and $(v_n)$ will be equivalent. 

Note that our estimates will imply that
$\|\sum_{i=1}^k (u_{n_i}-v_{n_i}) \| \leq \epsilon\| \sum_{i=1}^k (u_{n_i}+v_{n_i})\|$
whenever $k<n_1< \cdots< n_k$ and $k$ is large enough with respect to  $\epsilon$.  Thus we recover the result of saturation of $GM$ with finite block-sequence $(y_i)_{i=1}^{2k}$ such that
$\|\sum_{i=1}^{2k} (-1)^i y_i\| \leq \epsilon \|\sum_{i=1}^{2k} y_i\|$, for some $k=k(\epsilon)$ large enough, but of course our result is much stronger, since we can choose $(y_i)$ to be any finite subsequence $(u_{n_1},v_{n_1},\ldots,u_{n_k},v_{n_k})$ as above. This estimate implies that ${\mathcal G \mathcal M}$ does not contain a subspace with
an asymptotically unconditional basis, which  means by Gowers' dichotomy that ${\mathcal G \mathcal M}$ has a subspace which is HI. 
(and even, by \cite{W}, satisfies the HI property in a "uniform" way).

\subsection{Some preliminary definitions}\label{S:0}

\

We use the  usual definitions and notation for  $\coo$, $(e_i)$,  $E(x)$, $\supp(x)$, $\ran(x)$,   $E<F$ and $x<y$ for $E,F\subset \N$, and $x,y  \in\coo$.
The closed linear span of a basic sequence $(x_n)_{n \in \N}$ is denoted $[x_n, n \in \N]$.

We say that two vectors $x$ and $y$ in  $\coo$ {\em have the same distribution } and write
$x\eqdist y$ if there  there are natural numbers $l$ , $m_1<m_2 <\ldots m_l$,
and $n_1<n_2<\ldots n_l$, and a sequence $(a_i:i=1,2\ldots l)\subset \R$, so that
$$x=\sum_{i=1}^l a_i e_{m_i}\text{ and }  y=\sum_{i=1}^l a_i e_{n_i}.$$
We say $x$ {\em is the distribution of $y$} if $x$ and $y$ have the same distribution and if
the support of $x$ is an initial interval of $\N$. 

Note  that a vector $x\in\coo$ is uniquely defined by its distribution and its support.

 \begin{defin}\label{D:0.1} Let $X$ be a Banach space with a basis $(e_i)$. We  call  a vector $x$ in  $X$ an {\em $\ell_1^{+n}$-average,} 
  if $x=\frac1n\sum_{i=1}^n x_i$, where  $(x_i)_{i=1}^n$  is a block sequence (of $(e_i)$) in $B_X$.
 For $c\in (0,1]$  an $\ell_1^{+n}$-average $x$ is called {\em $\ell_1^{+n}$-average of constant $c$,} 
 if $\|x\|\ge c$.

   If moreover $(x_i)_{i=1}^n$ is $\frac1c$-isomorphic to the $\ell_1^n$ unit vector basis, we 
   say that $x$ is an    {\em $\ell_1^{n}$-average  of constant $c$.}
In particular it follows in that case that $\|\sum_{i=1}^n \pm x_i\|\ge c$.  \end{defin}

 \begin{remark} For a certain minor technical reason, we are not assuming  in Definition  \ref{D:0.1} that the sequence $(x_i)$ is normalized. But
 of course if  $x$ is supposed to be an $\ell_1^{+n}$- or an     $\ell_1^{n}$-average, of a constant $c$ close to $1$, then the norm of most of the $x_i$ also 
 has to be close to $1$.    \end{remark}

\section{The space $S$}\label{S:1}
We recall the space introduced in  \cite{Sch1}. We define
\begin{equation}
\label{E:1.1}
f(x)=\log_2(x+1),  \text{ for $x\ge 1$.}
\end{equation}

The space  $S$ is the completion of $\coo$ under the norm $\|\cdot\|_S$ which satisfies the following implicit equation.
\begin{equation}\label{E:1.2}
\|x\|_S=\max\Big(\|x\|_\infty,\max_{\begin{matrix}\scriptstyle l\in\N\\ \scriptstyle E_1<E_2<\ldots E_l\end{matrix}} \frac1{f(l)}\sum_{j=1}^l \|E_i(x)\|_S\Big) \text{ for }x\in c_{00} .
\end{equation}

As observed in \cite{Sch1}, there is a norm $\|\cdot\|_S$ on $\coo$, which satisfies   Equation 
\eqref{E:1.2}, the  completion $S$ of $(c_{00},\|\cdot\|_S)$ is reflexive, and $(e_i:i\in\N)$
is a $1$-subsymmetric  (i.e $1$-spreading and $1$-unconditional) basis of $S$.

For $l=2,3\ldots $ and  $x\in S$ we define
$$\|x\|_l=\frac1{f(l)} \max_{E_1<E_2<\ldots E_l}  \sum_{j=1}^l \|E_i(x)\|_S. $$
Then $\|\cdot\|_l$ is an equivalent norm on $S$ and for $x\in S$,
\begin{align}\label{E:1.3} 
&\frac1{f(l)} \|x\|\le \|x\|_l\le \|x\| \text{ and } \\
\label{E:1.4} &\|x\|=\max \big( \|x\|_\infty, \sup_{l\in\N} \|x||_l\big).
\end{align}

\subsection{Upper bounds of $\|\cdot\|_S$}

We will need to show some upper estimates for $\|\cdot\|_S$ and for  basic sequences which have spreading models equivalent to the unit basis in $S$.
\begin{defin}\label{D:1.5}  
For  a bounded sequence $(\xi_i)$  in $\R$ we denote the decreasing rearrangement of $(|\xi|_i|)$ by
$(\xi^\#_i)$. 

Assume that  $g:[1,\infty)\to [1,\infty)$ is an increasing  function   with  $g(1)=1$.
We define the following two norms on $c_{00}$. For $x=(x_i)\in c_{00}$ we define
\begin{align*}
\|x\|_g&=\max_{n_1<n_2<\ldots n_l, l\in\N} \frac1{g(l)} \sum_{i=1}^l |x_{n_i}|\text{ and }
\tn x\tn_g=\sum_{i=1}^\infty \frac1{g(i)}x^\#_i.
\end{align*}\end{defin}
It is clear that $\|\cdot\|_g\le \tn\cdot\tn_g$. The following Lemma describes a situation in which 
 we can bound $\tn\cdot\tn_{g^p}$, $1<p<q$ by   a multiple of $\|\cdot\|_{g^q}$.

 \begin{lem}\label{L:1.6} For $0<p<q$ there is a constant $C(p,q)$ so that 
 $$\tn x\tn_{f^{q}}\le C(p,q)\|x\|_{f^p}, \text{ for all $x\kin c_{00}$}.$$
 Here $f:[0,\infty)\to [0,\infty)$ is  defined as in \eqref{E:1.1} by $f(x)=\log_2(x+1)$, for $x\ge 1$.
 \end{lem}
 \begin{proof} We first  observe that 
 $$C(p,q)=\sum_{n=1}^\infty \frac{f^p(n)-f^p(n-1)}{f^q(n)}<\infty.$$
 Indeed, by the Mean Value Theorem, there is for every $n\kin\N$ an $\eta_n\kin(n,n\!+\!1)$, so that 
 $$ f^p(n)-f^p(n-1)=\big( \log_2(e)\big)^p\frac{\partial\ln^p(x)}{\partial x}\Big|_{x=\eta_n} = \big( \log_2(e)\big)^p p\frac{\ln^{p-1}(\eta_n)}{\eta_n},$$
 and thus
$$\big( \log_2(e)\big)^{q-p} \sum_{n=2}^\infty \frac{f^p(n)-f^p(n-1)}{f^{q}(n)} =\sum_{n=2}^\infty \frac1{\eta_n} \frac1{\ln^{1-p}(\eta_n)\ln^q(n+1)}
<\sum_{n=2}^\infty \frac1{n} \frac1{\ln^{1+q-p}(n)},  $$
which is finite  by the integral test.

Secondly we claim that for $L\in\N$ 
$$M_L=\max\{\tn x\tn_{f^q} : \ran(x)\subset [1,L],  \text{ and }\|x\|_{f^p}\le 1\},$$
is achieved for the vector
$$x^{(L)}=\sum_{j=1}^L (f^p(j)-f^p(j-1))e_j \text{ with $f(0)=0$,}$$
which would imply that $M_L\le C(p,q)$.
Indeed, $\|x^{(L)}\|_{f^p}=1$, and if $z=(z_i)_{i=1}^L\in \R^L$, $\|z\|_{f^p}=1$, and
$$ \tn z\tn_{f^q}= \sum_{j=1}^L \frac{z^\#_j}{f^q(j)}=M_L,$$
we can assume without loss of generality that $z_1\ge z_2\ge \ldots z_L\ge 0 $.
Note that actually $z_L>0$. Otherwise let 
$l_0=\min\{j:  z_i=0  \text{ for all } i\ge j \}$, and note for $l\ge l_0$
that 
$$\frac1{f^p(l)}\sum_{j=1}^{l} z_j=\frac1{f^p(l)}\sum_{j=1}^{l_0-1} z_j
<\frac1{f^p(l_0-1)}\sum_{j=1}^{l_0} z_j\le 1.$$
Thus we could  increase the value of $z_{l_0}$, and thus increase the value of $\tn z\tn_{f^q}$, without increasing the value of
$\|z\|_{f^p}$, which contradicts the maximality of $z$.

We want to show now that  $z=x^{(L)}$, which would imply our claim.
If this were not true we put
$$l_o=\min \big\{ j \kin\{1,2,\ldots,L\}: z_j\not=f^p(j)-f^p(j-1)\big\}.$$
First we note that $z_{l_0}<f^p(j)-f^p(j-1)$, because otherwise $z_{l_0}>f^p(j)-f^p(j-1)$ and thus
$\frac1{f^p(l_0)}\sum_{j=1}^{l_0} z_j >\frac1{f^p(l_0)}\sum_{j=1}^{l_0}f^p(j)-f^p(j-1)=1$,

Note that $l_0\not=L$ otherwise we could increase $z_{L}$ to $f^p(L)-f^p(L-1)$, which would not 
increase $\|z\|_{f^p}$, but certainly increase $ \tn z\tn_{f^q}$.
If $l_0<L$ we could  increase $z_{l_0}$ by    
$\min(f^p(l_0)-f^p(l_0-1), z_{l_0})>0$ and decrease $z_{l_0+1}$ by the same amount. This would not increase the
 $\|\cdot\|_{f^q}$-norm but it would increase the $\tn\cdot \tn_{f^p}$-norm of $z$.
 \end{proof}
 The proof  of the next  Lemma could be shown using \cite[Theorem 1.1]{Sch2} and its proof. Nevertheless, since the arguments
in this case are much simpler we prefer to present a self contained argument.

\begin{lem}\label{L:1.8} 
Let  $(x_n)$ and $(y_n)$ be two basic seminormalized weakly null  sequences in a Banach space $X$, having spreading models $E$ and $F$ with bases  $(\xt_n)$ and $(\yt_n)$, respectively.
 Assume that for some $0<p<q$ and some $0<c,C<\infty$ it follows that
\begin{equation}\label{E:1.8.1}
\Big\|\sum_{i=1}^\infty a_n \xt_n\Big\|_E \ge  c\Big\|\sum_{i=1}^\infty a_n e_n\Big\|_{f^p} \text{ and }
\Big\|\sum_{i=1}^\infty a_n \yt_n\Big\|_F \le  C\Btn\sum_{i=1}^\infty a_n e_n\Btn_{f^q}.
\end{equation}
Then there  is a subsequence $(n_k)$ of $\N$ so that the map $x_{n_k}\mapsto y_{n_k}$ extends to a linear bounded operator.
\end{lem}
\begin{remark} Using the arguments in \cite{Sch2} one can actually show that under the assumption of Lemma
\ref{L:1.8} there  is a subsequence $(n_k)$ of $\N$ so that the map $x_{n_k}\mapsto y_{n_k}$ extends to a linear bounded and strictly singular operator.
\end{remark}

Before proving Lemma \ref{L:1.8} we will need the following
\begin{lem}\label{L:1.7} Assume $0<p<q$  and define for $\vp>0$ 
\begin{equation}\label{E:1.7.1}
\Delta_{(p,q)}(\vp)= \sup \left\{ \Big\|\sum_{i=1}^\infty a_i  e_i \Big\|_{f^q}:   |a_i|\le \vp,\, i=1,2\ldots , \text{ and }  \Big\|\sum_{i=1}^\infty a_i  e_i \Big\|_{f^p} \le 1\right\}.
\end{equation}
Then 
\begin{equation}\label{E:1.7.2}
\lim_{\vp\searrow 0} \Delta_{(p,q)}(\vp)=0.
\end{equation}
\end{lem}
\begin{proof}
Let $\eta>0$ be arbitrary and choose $n_\eta\in\N$ so that 
$\frac1{f^{q-p}(n)}\le \eta$, for all $n\le n_\eta$, and then choose 
$$\vp=\eta \min_{n\le n_\eta} \frac{f^q(n)}n.$$
For any $(a_i)\in c_{00}$, with $|a_i|\le \vp$, for $i\kin\N$, and $\|\sum_{i=1}^\infty a_i e_i\|_{f^p}\le 1$ it follows therefore that, for some choice of $n\in\N$ and $i_1<i_2<\ldots i_n$ in $\N$,
 we have
$$
\Big\|\sum_{i=1}^\infty a_i e_i\Big\|_{f^q}=
\frac1{f^q(n)} \sum_{s=1}^n |a_{i_s}|
\le  \begin{cases} \vp \frac{n}{f^q(n)} \le \eta  &\text{ if $n\le n_\eta$}\\
   \frac1{f^{q-p}(n)}  \|\sum_{i=1}^\infty a_i e_i\|_{f^p}\le \eta &\text{ if $n>n_\eta$,} 
    \end{cases}
$$
which verifies our claim.
\end{proof}

\begin{proof}[Proof of Lemma \ref{L:1.8}]
We can assume that $(\yt_n)$ is not equivalent  to the $c_0$ unit vector basis. Otherwise we may replace
the norm on $[y_n:n\in\N]$ by
$$\Btn \sum a_i y_i\Btn=\Big\| \sum a_i y_i\Big\|+ \Big\|\sum a_ie_i\Big\|_{f^q} \text{ if $(a_i)\in c_{00}$}.$$
We can therefore assume that for every $\vp>0$
the number 
$$l(\vp)=\max\left\{ l: \begin{matrix} \text{ there are } (a_i)_{i=1}^l,\,\,|a_i|\ge \vp,\,  i=1,2 \ldots l,\,     \text{  and }    \\
                 n_1<n_2<\ldots n_l, \text{ so that } \,\,\Big\|\sum_{i=1}^l a_i y_{n_i}\Big\|\le 1 \end{matrix}\right\}$$
exists.

Let $r=(p+q)/2$.
By Lemma \ref{L:1.7}
we can choose a sequence $(\vp_n)\subset (0,1)$ so that
\begin{equation}\label{E:1.8.2} 
\sum_{n\in\N} \Delta_{(p,r)}(c\vp_n/6) \le 1 \text{ and } \sum_{n=1}^\infty n\vp_n \le 1.
\end{equation}

Using the {\em Schreier unconditionality} of basic sequences  \cite{Od} (see also \cite{DOSZ} for a more general statement), the fact that $(\xt_n)$ is the spreading model
of $(x_n)$,  and our assumption \eqref{E:1.8.1},
we can assume, after passing to  simultaneous  subsequences  of $(x_n)$ and $(y_n)$,  if necessary, that
 for all $(a_i) \in c_{00}$ and all finite $F\subset \N$, with $n\le \min F$  and $\#F \le  l(\vp_{n+1})$ we have 
 \begin{align}\label{E:1.8.4} 
          \Big\|  \sum_{i\in F}  a_i  e_i\Big\|_{f^p}     & \le       \frac1c  \Big\|  \sum_{i\in F}  a_i \xt_i\Big\|   \le \frac2c \Big\|\sum_{i\in F} a_i x_i\Big\|\le \frac6c \Big\| \sum_{i=1}^\infty a_i x_i\Big\|,
\intertext{and by using the fact that $(\yt_n)$ is the spreading model
of $(y_n)$,  our assumption \eqref{E:1.8.1},  and Lemma \ref{L:1.6},  we can assume that  for some constant $C_3$ and 
 for all finite $F\subset \N$, with $n\le \min F$  and $\#F \le  l(\vp_{n+1})$,  and all $(a_i)_{i\in F}$ we have }
   \label{E:1.8.5}  \Big\|  \sum_{i\in F}  a_i  y_i\Big\|  &\le    2\Big\|  \sum_{i\in F}  a_i  \yt_i\Big\| \le 2C     \Btn  \sum_{i\in F}  a_i  \Btn_{f^q} \le C_3 \Big\|  \sum_{i\in F}  a_i  e_i\Big\|_{f^r}.
\end{align}

By Elton's {\em near unconditionality}  \cite{El} (see also \cite[Theorem 6]{DOSZ})  and the fact that
 $l(\vp_1)$ is finite we can assume, after passing to subsequences,  if necessary,
that there are constants $C_1$ and $C_2$
so that for every $(a_i)\in c_{00}$, with $\big\|\sum a_i x_i \big\|\le 1$, it follows that
\begin{equation}\label{E:1.8.6}
\Big\| \sum_{j=1, |a_j| > \vp_1 }^\infty a_j y_j\Big\|\le C_1\Big\| \sum_{j=1, |a_j| > \vp_1}^\infty a_j x_j\Big\|\le C_2\Big\| \sum_{j=1}^\infty a_j x_j\Big\|\le C_2.
\end{equation}

Now let $(a_j)\in c_{00}$ and assume that $\big\|\sum_{i=1}^n a_i x_i\| =1$.
Then, by \eqref{E:1.8.6}, 
$$\Big\| \sum_{j=1}^\infty a_j y_j\Big\|\le \Big\| \sum_{ |a_j|\le \vp_1}a_j y_j\Big\|+  \Big\| \sum_{|a_j|> \vp_1} a_j y_j\Big\|
\le  \Big\| \sum_{ |a_j|\le \vp_1}^\infty a_j y_j\Big\|+ C_2 \Big\| \sum_{j=1}^\infty a_j x_j\Big\|
$$
and 
\begin{align*}
\Big\| \sum_{j=1, |a_j|\le \vp_1}^\infty a_j y_j\Big\| 
&\le \sum_{n=1}^\infty \Big\|\sum_{\vp_{n+1} <|a_j|\le \vp_n }a_j y_j\Big\|\\
&\le   \sum_{n=1}^\infty \Bigg[ n\vp_n + \Big\|\sum_{n<j, \vp_{n+1} <|a_j|\le\vp_n} a_j y_j\Big\|\Bigg]\\
&\le 1+ C_3 \sum_{n=1}^\infty\Big\|\sum_{n<j, \vp_{n+1} <|a_j|\le\vp_n} a_j e_j\Big\|_{f^r} \text{ (by \eqref{E:1.8.2} and  \eqref{E:1.8.5})}\\
&= 1+ \frac{6C_3}c\sum_{n=1}^\infty \Big\|\sum_{n<j, \vp_{n+1} <|a_j|\le\vp_n} \frac{c}6a_j e_j\Big\|_{f^r}.
\intertext{ Now it follows from \eqref{E:1.8.4} that 
$$  \Big\|\sum_{n<j, \vp_{n+1} <|a_j|\le\vp_n}  \frac{c}6  a_j e_j\Big\|_{f^p} \le 1$$ 
and thus \eqref{E:1.8.2}  and the definition of $\Delta_{p,r}$ yield that}
\Big\| \sum_{j=1, |a_j|\le \vp_1}^\infty a_j y_j\Big\| 
&\le 1+\frac{6C_3}c .
\end{align*}
We proved therefore that if $\big\|\sum_{i=1}^\infty a_i x_i\big\|\le 1$ then
$$\Big\|\sum_{i=1}^\infty a_i y_i\Big\|\le C_2+ 1+  \frac{6C_3}c,$$
which finishes the proof of our claim.
\end{proof}

We finally want to compare the norms $\tn\cdot\tn_f$ and $\|\cdot\|_S$ and first prove the following Lemma.

\begin{lem}\label{L:1.9}
 For every $x^*=(\xi_j)\in B_{S^*}$ and $n\in\N$, we have  that
 \begin{equation}\label{E:1.9.1}
 \xi_n^\#\le \frac1{f(n)}.
 \end{equation}
 \end{lem}
  \begin{proof}
By the 1-unconditionality of both norms in $S$ and $S^*$  
we need to prove \eqref{E:1.9.1} only for
non negative sequences $x^*=(\xi_j)$ in $c_{00}$.  
Let $E=\{j_1, j_2,\ldots, j_n\}\subset \N$ have $n$ elements,
so that 
$$\xi_{j_1}\ge \xi_{j_2}\ge \ldots \xi_{j_n}=\xi_{n}^{\#},$$   
and put
$$y^*=\xi_n^\#\sum_{s=1}^n  e^*_{j_s}.$$
Since $S^*$ is supression 1-unconditional, it follows that 
$\|y^*\|\le 1$, and since
$$y=\frac{f(n)}{n}\sum_{s=1}^n e_{j_s}\in S_S,$$
(see \cite{Sch1}) it follows that
$$1\ge \la y^*,y\ra = f(n)\xi_n^\#,$$
which proves our claim.
\end{proof}

\begin{cor}\label{C:1.10} For $x\in c_{00}$ we have
\begin{equation}\label{E:1.10.1}
\|x\|_S\le \tn x\tn_f.
\end{equation}
\end{cor}

\subsection{Yardstick vectors}

\

The following type of vectors were introduced in \cite{KL}.
\begin{defin}\label{D:1.11} (Yardstick Vectors)

 We call a finite or infinite sequence of natural number $m_1,m_2, m_3, \ldots $ {\em admissible},
  if  for any  $i$, for which $m_i$ exists, $m_{i}$ is even and is  a  multiple of the product
   $\prod_{A\subset{1,2,\ldots i-1} }\Big(\sum_{j\in A} m_j \Big)$  (as usual $\prod_{\emptyset} =1$).
Note that any subsequence of admissible sequences is also admissible.

By induction we define the vector $y (m_1,m_2,\ldots m_k)$ for each $k$ and each  admissible finite sequence
$(m_1,m_2,\ldots,m_k)\subset \N$;  the support of   $y (m_1,m_2,\ldots m_k)$ will be the interval $[1,\sum_{i=1}^k m_i]$.

If $k=1$ we put for $m\in\N$
$$y (m)= \frac{f(m)}{m}  \sum_{i=1}^m e_i.$$
Assume  that 
$y (m_1,m_2,\ldots m_{k'})$ has been defined for each $k'<k$ and each admissible  sequence
$(m_1,m_2,\ldots,m_{k'})\subset \N$.

From our induction hypothesis the  support of $y (m_1,m_2,\ldots, m_{k-1})$  is $[1,m_1+m_2+\ldots +m_{k-1}]$ and we write   $y (m_1,m_2,\ldots, m_{k-1})$ as
 $$y(m_1,m_2, \ldots, m_{k-1}) =\sum_{i=1}^{m_1+m_2+\ldots+m_{k-1}} a_i e_i.$$
Now we define $\tilde y$ to be the vector, which has the same distribution as 
 $y(m_1,m_2,\ldots,, m_{k-1}) $, and whose support
 is 
 $$\supp(\tilde y)= 
 \left\{1+ (i-1) \frac{m_{k}}{m_1+m_2+\ldots m_{k-1}} : i=1,2\dots, \sum_{j=1}^{k-1} m_j\right\} $$
(i.e we spread out the coordinates of $y(m_1,m_2, m_{k-1}) $, so that between any two successive  non zero coordinates there are  $\frac{m_{k}}{m_1+m_2+\ldots m_{k-1}}$ zeros).

 Then we  define 
 $$y(m_1,m_2,\ldots, m_{k}) =\tilde y +\frac{f(m_k)}{m_k} 
 \sum_{i\in[1, m_1+m_2++\ldots m_k]\setminus\supp(\tilde y)}e_i$$
(i.e. we are replacing the zeros on the interval $[1,m_1+m_2+\ldots m_k]$ by the value $f(m_k)/m_k$)).  
 So,  for example, $y(m_1)$ and $y(m_1,m_2)$ are the following vectors:
 \begin{align*}
 y(m_1)&=\underbrace{\Big(\frac{f(m_1)}{m_1},\frac{f(m_1)}{m_1},\ldots,\frac{f(m_1)}{m_1}\Big)}_{m_1\text{ times}}\\
 y(m_1,m_2)&=\Big(\underbrace{\frac{f(m_1)}{m_1}, 
 \underbrace{\frac{f(m_2)}{m_2},\ldots,\frac{f(m_2)}{m_2}}_{m_2/m_1\text{ times}},\ldots,
 \frac{f(m_1)}{m_1}, 
 \underbrace{\frac{f(m_2)}{m_2},\ldots,\frac{f(m_2)}{m_2}}_{m_2/m_1\text{ times}}}_{ m_1\text{times}}\Big).
 \end{align*}

If $\xb=(x_n)_{n\in \N}$ is a block sequence in $\coo$, and if $(m_1,\ldots m_k)\subset \N$ is admissible, we define
$y_{\xb}(m_1,m_2,\ldots,m_k)  $ to be a linear combination of the $x_n$' s with the same distribution as $y(m_1,m_2,\ldots,m_k)$ has on the $e_n$'s,  i.e.
$$y_{\xb}(m_1,m_2,\ldots,m_k)=\sum_{i\in\supp(y(m_1,m_2,\ldots,m_k))} a_i x_i,$$
where the $a_i$ are such that
$$y(m_1,m_2,\ldots,m_k)=\sum_{i\in\supp(y(m_1,m_2,\ldots,m_k))} a_i e_i.$$
\end{defin}
It  follows from the arguments in \cite{KL}  that for $k\in\N$ and $\vp>0$  one can find $m_1<m_2<\ldots m_k$ in $\N$ so that 
$\|y(m_1,m_2,\ldots m_k)\|_S\le 1+\vp$.  Since $y(m_1,m_2,\ldots m_k)$ is the sum of disjointly supported  vectors $z_1, z_2,\ldots z_k$, 
with   $z_i$ having the  same distribution as $\frac{f(m_i)}{m_i}\sum_{j=1}^{m_i}e_j$,  for $i=1,2\ldots k$,  (and thus $\|z_j\|=1$, by \cite{Sch1}),
it follows that $\ell_\infty^k$, $k\in\N$ are uniformly represented in $S$. Something stronger is true. Using similar arguments  as in \cite{KL} it is actually possible
 to prove under appropriate growth conditions on $(m_i)$ that the sequence
 $\big(y(m_1,m_2,\ldots m_k):k\kin\N\big)$ is uniformly bounded in $S$.  For completeness we will present a self contained proof of this fact. 
First we prove the following lemma,  which will serve as the induction step for choosing the sequence $(m_i)$.

\begin{lem}\label{L:1.12}
 Assume we are given $k,m\in\N$, with $k<m$, $c\ge 1$ and some $\vp\in(0, f(2)-1)/f(2))$ satisfying the following conditions:
 \begin{align}\label{E:1.12.1}
 &\text{$m$ is divisible by $k$}\\
\label{E:1.12.2} &f(m)\ge C \max \left(  \frac{50}{\vp^2},      \frac{f(l_0)f(l_0-1)}{f(l_0)-f(l_0-1)}\right)\\
\intertext{where $l_0=\min\{l\in\N: f(l)\ge 6\}$}
\label{E:1.12.3} &\frac{f(m)}{f(m/k)}\le 1+\frac\vp6.
 \end{align}
 Assume  further $(j_s)_{s=1}^m\subset \N$ and $(x_s)_{s=1}^m\subset S$ have  the property that
 \begin{align}\label{E:1.12.4}
 &e_{j_1}< x_1<e_{j_2}< x_2 \ldots e_{j_{m-1}}<x_{m-1}<e_{j_m}<x_m, \text{ and }\\
 \label{E:1.12.5}
 &\|x_s\|\le \frac{C}{m}, \text{ for $s=1,2\ldots m$.}
 \end{align} 
 Then it follows that
 \begin{align}\label{E:1.12.6}
 & \Big\|\sum_{s=1}^m \frac{f(m)}{m} e_{j_s} + x_s\Big\|_S\le  {C(1+\vp)}\text{ and }\\
  \label{E:1.12.7}&\Big\|\sum_{s=(i-1)(m/k) +1}^{im/k} \frac{f(m)}{m} e_{j_s} + x_s\Big\|_S \le \frac{C(1+\vp)}{k} \text{ for $i=1,2,3,\ldots,k$}.
 \end{align}
 In particular the vectors
 $$y_i= \frac{k}{C(1+\vp)} \sum_{s=(i-1)(m/k) +1}^{im/k}\frac{f(m)}{m} e_{j_s} + x_s,\text{  for $i=1,2,\ldots ,k$}$$
 are in $B_S$ and $(y_i)$ is $C(1+\vp)$ -equivalent to the unit vector basis in $\ell_1^k$.
 
 \end{lem}

\begin{proof} We note that for any scalars $(a_i)_{i=1}^k$, we have
$$\big\|\sum_{i=1}^k a_i y_i\big\| \geq f(m)^{-1}\big(\sum_{s=1}^m e_{j_s}^*\big)\big(\sum_{i=1}^k a_i y_i\big)
=f(m)^{-1}\sum_{i=1}^k a_i \Big(\sum_{s=(i-1)(m/k) +1}^{im/k} e_{j_s}^*(y_i)\Big).$$
 It follows easily, assuming \eqref{E:1.12.7} and using the $1$-unconditionality of the basis, that
$(y_i)$ is $C(1+\vp)$ -equivalent to the unit vector basis in $\ell_1^k$.

To prove \eqref{E:1.12.6} and \eqref{E:1.12.7} we put 
$$x=\sum_{s=1}^m \frac{f(m)}{m} e_{j_s} + x_s.$$
We will proof by induction for each $n\kin\{1,2,\ldots m\}$,  that whenever $0\kleq s_0<s_1\kleq m$, with $s_1-s_0=n$, and
$I\subset\N$ is an interval with $j_{s_0}<I<j_{s_1+1}$ (where we let  $j_0=0$ and $j_{s_{m+1}}=\infty$), then 
\begin{equation}\label{E:1.12.8}
\| I(x)\|\le 
\frac{f(m)}{m} \frac{n}{f(n)} C\Big( 1+\frac{\vp}3 \Big)+\frac{n}{m}  C\frac{\vp}3.
\end{equation}
From that we deduce \eqref{E:1.12.6} by letting $I\!=\!\N$ and $\!n=\!m$. Moreover, if we 
put $I\!=\![j_{(i-1)(m/k)+1}, j_{i(m/k)+1}\!-\!1]$,   for $i=1,2\ldots k$, we deduce from \eqref{E:1.12.8} and  \eqref{E:1.12.3} that
\begin{align*}
\Big\|\sum_{s=(i-1)(m/k) +1}^{im/k} \frac{f(m)}{m} e_{j_s} + x_s\Big\|
 &= \|I(x)\|\\
&\le \frac{f(m)}{m} \frac{m/k}{f(m/k)} C\Big( 1+\frac{\vp}3 \Big)+\frac{m/k}{m}  C\frac{\vp}3\\
&=\frac1k \frac{f(m)}{f(m/k)} C\Big( 1+\frac{\vp}3 \Big)+\frac1k  C\frac{\vp}3\le \frac1kC(1+\vp)
\end{align*}
which implies \eqref{E:1.12.7}.

First let $n\in\N$, so that $f(n)\le\frac6{\vp}$, and let  $I\subset \N$ be an interval  with $j_{s_0}<I<j_{s_1+1}$ for some choice of  $s_0,s_1\in\{1,2\ldots,m\}$, and   $s_1-s_0=n$, $l\ge 2$.
Then
\begin{align*}
\|I(x) \|_l &\le \Big\|\sum_{s=s_0+1}^{s_1}\frac{f(m)}{m} e_{j_s}\Big\| + \Big\|\sum_{s=s_0}^{s_1}  x_{s}\|\\
&\le  \frac{f(m)}{m} \frac{n}{f(n)} +C\frac{n+1}{m} \\
&=    \frac{f(m)}{m} \frac{n}{f(n)} \Big[1+ C\frac{n+1}{n}\frac{f(n)}{f(m)}\Big]\\
 &\le            \frac{f(m)}{m} \frac{n}{f(n)} C \Big[1+  \frac{12}{\vp f(m)}\Big]   \le         \frac{f(m)}{m} \frac{n}{f(n)} C \Big[1+  \frac{\vp}{3}\Big]   \text{ (by \eqref{E:1.12.2} )}
\end{align*}
 which implies our claim for $n\in\N$, for which   $f(n)\le\frac6{\vp}$.

Assume  that our induction hypothesis is true for all $n'<n$,  $n\in\N$, with $f(n)>6/n$,  
 and  all intervals $I\subset \N$  for which there are  $j_{s_0}<I < j_{s_{1}+1}$ with $0\le s_0<s_1\le m$ and $s_1-s_0=n$.

Let $l\in\N$, $l\ge 2$,  such that $\|x\|=\|x\|_l$ (since $n\ge 2$ it follows that $\|I(x)\|_\infty< \|I(x)\|_2$, and thus $l\ge 2$).  

We choose numbers $l_1$ and $l_2$ in $\N\cup\{0\}$, with $l=l_1+l_2$,
 and intervals $E_1^{(1)}<E_2^{(1)},\ldots E^{(1)}_{l_1}$ and $E_1^{(2)}<E_2^{(2)},\ldots E^{(2)}_{l_1}$,
 so that 
 $$\bigcup_{t=1}^{l_1} E^{(1)}_t\cap \bigcup_{t=1}^{l_1} E^{(2)}_t=\emptyset \text{ and }
\bigcup_{t=1}^{l_1} E^{(1)}_t\cup \bigcup_{t=1}^{l_1} E^{(2)}_t=I,$$ 
and so that each of the $E^{(1)}_t$ contains 
at least one of the $j_s$, $s_0<s\le s_1$ and none of the    $E^{(2)}_t$ intersects with $\{j_{s_0+1}, j_{s_0+2},\ldots j_{s_1}\}$,
and so that 
\begin{equation}\label{E:1.12.9} 
\| I(x)\|_l =\frac1{f(l)}\left(\sum_{t=1}^{l_1} \|E^{(1)}_t(x)\|+ \sum_{t=1}^{l_2} \|E^{(2)}_t(x)\|\right).
\end{equation}
We note that $l_1\ge2$, otherwise it would follow that $l_1=1$ and  for all $t=1,2\ldots l_2 $
either $j_{s_0}< E^{(2)}_t<j_{s_0+1}$ or $j_{s_1}< E^{(2)}_t<j_{s_1+1}$, and thus, by \eqref{E:1.12.5}
$$
\|I(x)\|=\|I(x)\|_l \le \frac{2C}{m} + \frac1{f(l)} \|E^{(1)}_1(x)\|\le \frac{2C}{m} + \frac1{f(2)} \|I(x)\|,$$
and thus
$$\frac{f(m)}{m}\Big(1- \frac1{f(l)} \Big) \le \frac{2C}{m},$$
which contradicts \eqref{E:1.12.2} and the restrictions on $\vp$.

We can therefore apply our induction hypothesis and deduce that there are numbers 
$s_0=\st_0<\st_1<\ldots \st_{l_1}=s_1$ 
 so that for $t=1,2,\ldots l_1$ 
 \begin{equation}\label{E:1.12.10} 
 \|E^{(1)}_t(x)\|\le \frac{f(m)}{m} \frac{\st_t-\st_{t-1}}{  f(\st_t-\st_{t-1})} C\Big(1+\frac{\vp}3\Big)+ \frac{\st_t-\st_{t-1}}{m} C\frac{\vp}3.  
 \end{equation}
 Moreover it follows that
 \begin{equation}\label{E:1.12.11} 
\frac1{f(l)}\sum_{t=1}^{l_2} \|E^{(2)}_t(x)\|\le \Big\|\sum_{s=s_0}^{s_1} x_s\Big\|\le C\frac{n+1}{m}.
 \end{equation} 
 
 \noindent
 Case 1. If $ 6 \le f(l)$, then we deduce that
 \begin{align*}
 \| I(x)\|_l &=\frac1{f(l)}\left(\sum_{t=1}^{l_1} \|E^{(1)}_t(x)\|+ \sum_{t=1}^{l_2} \|E^{(2)}_t(x)\|\right)\\
 &\le  C\frac{n+1}{m}+ \frac1{f(l)}  C\sum_{t=1}^{l_1} \Big[ \frac{f(m)}{m}\frac{\st_t-\st_{t-1}}{  f(\st_t-\st_{t-1})} \Big(1+\frac{\vp}3\Big)+ \frac{\st_t-\st_{t-1}}{m} \frac{\vp}3\Big] \\
 &\le C\frac{n+1}{m}+\frac1{f(l)}   C\Bigg[\frac{f(m)}{m}\frac{n}{f(n/l_1)}  \Big(1+\frac{\vp}3\Big)+ \frac{n}{m}\frac{\vp}3\Bigg] \\
 &\text{(By the concavity of the map $\xi\mapsto \xi/f(\xi)$)}\\
&= \frac{f(m)}{m} C\frac{n}{f(l) f(n/l_1)}  \Big(1+\frac{\vp}3\Big)+     C \frac{n}{m} \Big[1+   \frac{1}{n}  + \frac{\vp}{3f(l)}\Big] \\
&\le  \frac{f(m)}{m} C\frac{n}{f(n)}  \Big(1+\frac{\vp}3\Big) + C\frac{n}{m} \Big(1+\frac\vp3\Big)
 \end{align*}
where the last inequality  follows from the fact that $f(a/b) f(b)\ge f(a)$ for $a,b\ge 2$ (see \cite{Sch1}) and  \eqref{E:1.12.2}. This finishes the proof
of our induction step in this case.
 
\noindent
Case 2. If $f(l)<6$ we claim that $l_2=0$. Indeed, otherwise  $l=l_1+l_2\ge 3$ (we already observed that $l_1\ge 2$) and
 \begin{align*}
 \|I(x)\|_{l-1} &\ge \frac1{f(l-1)} \Bigg[\sum_{t=1}^{l_1} \|E^{(1)}_t\| +\sum_{t=2}^{l_2} \|E^{(2)}_t(x)\|\Bigg]\\
                      &\ge   \frac1{f(l)} \Bigg[\sum_{t=1}^{l_1} \|E^{(1)}_t\| +\sum_{t=2}^{l_2} \|E^{(2)}_t(x)\|\Bigg] + \Bigg(\frac1{f(l-1)} -\frac1{f(l)}\Bigg)\frac{f(m)}{m}\\
                      & >  \frac1{f(l)} \Bigg[\sum_{t=1}^{l_1} \|E^{(1)}_t\| +\sum_{t=2}^{l_2} \|E^{(2)}_t(x)\|\Bigg] + \frac1{f(l)} \|E^{(2)}_1(x)\| \text{ (by  \eqref{E:1.12.2})}\\
                      &=\|I(x)\|_l,
 \end{align*}
 which contradicts the assumption that $\|I(x)\|=\|I(x)\|_l$.
 So it follows  that   $l=l_1$ and from \eqref{E:1.12.10} and the concavity of the map $\xi\mapsto f(\xi)/\xi$, $\xi\ge 1$  it follows therefore that
 \begin{align*}
 \|I(x)\|_l&\le  \frac1{f(l_1)} \sum_{t=1}^{l_1} \Big[\frac{f(m)}{m} \frac{\st_t-\st_{t-1}}{  f(\st_t-\st_{t-1})} C\Big(1+\frac{\vp}3\Big)+ \frac{\st_t-\st_{t-1}}{m} C\frac{\vp}3\Big]\\
 &\le \frac{f(m)}{m} \frac{n}{f(l)f(n/l)} C\Big(1+\frac{\vp}3\Big)+\frac{n}{m} C\frac{\vp}3\le  \frac{f(m)}{m} \frac{n}{f(n)} C\Big(1+\frac{\vp}3\Big)+\frac{n}{m} C\frac{\vp}3,
 \end{align*} 
 which finishes the proof of the inductions step the the proof of our lemma.
 \end{proof} 
 
 \begin{lem}\label{L:1.13} Assume that $(\vp_i)\subset (0,(f(2)-1)/f(2))$ is summable, and put $C_i=  \prod_{j\ge i}(1+ \vp_j)$, for $i\in\N\cap\{0\}$.
 
 Assume that the sequence $(m_i: i\in \N\cup\{0\})\subset \N$  is an admissible sequence and  satisfies the following growth conditions. For all $i\in\N$ we assume that
 \begin{align} \label{E:1.13.1} &f(m_i)\ge C_i \max \left( \frac{50}{\vp_i^2} ,\frac{f(l_0)f(l_0-1)}{f(l_0)-f(l_0-1)}\right), \\
\intertext{where $l_0=\min\{l\in\N: f(l)\ge 6\}$, and }\notag\\
 &\label{E:1.13.2} \frac{f(m_i)}{f(m_i/m_{i-1})}\le 1+\frac{\vp_i}6.     
  \end{align}
  Then it follows for all $i\le j$ in $\N$ that 
  \begin{align} \label{E:1.13.3}
   &\|y(m_i,m_{i+1},\ldots m_j)\| \le C_{i}, \text{ and} \\
    \label{E:1.13.4} &\frac1{C_{i}} y(m_i,m_{i+1},\ldots m_j) \text{ is an $\ell^{m_{i-1}}_1$-average of constant $1/C_{i}$.}  
  \end{align}
 \end{lem}
 \begin{remark} For the sequence $(m_i)$ as chosen in Lemma  \ref{L:1.13} we deduce therefore that, if $k\kin\N$ and $\vp\!>\!0$   and
 if $i_0\kin\N$ is chosen so that $k\kleq m_{i_0}$ and $\prod_{i=i_0}^\infty (1+\vp_i)\kleq 1+\vp$, then for all 
 sequences $i_0<i_1<i_1<\cdots< i_l$, $l\in\N$, it follows that
 \begin{align}\label{E:1.14}
 &\|y(m_{i_1},m_{i_2}, \ldots,m_{i_l})\|_S\le 1+\vp\text{ and }\\
 \label{E:1.15}
&\frac1{1+\vp} y(m_{i_1},m_{i_2}, \ldots, m_{i_l})\text{ is an $\ell^{m_{i_0}}_1$-average of constant $\frac1{1+\vp}$.}
 \end{align}
 \end{remark}

 \begin{proof}[Proof of Lemma \ref{L:1.13}] Let  the sequence $(m_j:j\kin\N\cup\{0\})$ be chosen as required. Let $j\kin\N$. By induction on $i=0,1,2\ldots, j-1$ we will show that
  \begin{align} 
    \label{E:1.13.5}&\|y(m_{j-i},m_{j-i+1},\ldots, m_j)\| \le C_{j-i}, \text{ and} \\
   \label{E:1.13.6}  &\frac1{C_{j-i}} y(m_{j-i},m_{j-i+1},\ldots m_j) \text{ is an $\ell^{m_{j-i-1}}$-average of constant $1/C_{j-i}$. }\\
   &\text{ More precisely, we can write $y=y(m_{j-i},m_2,\ldots m_j) $ as} \notag \\
     &y=\frac{C_{j-i}}{m_{j-i-1}} \sum_{s=1}^{m_{j-i-1}} y_s \text{ where $y_1<y_2 <\ldots y_{m_{j-i-1}} $ are in $B_S$,}\notag\\
     &\text{equally distributed and $C_{j-i}$-equivalent to the basis of $\ell_1^{m_{j-i-1}}$.}\notag
  \end{align}
   For $i=0$ it follows that $y(m_j)=\frac{f(m_j)}{m_j} \sum_{s=1}^{m_j} e_s \in S_S$, and the conditions of Lemma \ref{L:1.12} are satisfied with $m=m_j$, $\vp=\vp_j$, and 
   $C=1\le C_{j+1}$, $k=m_{j-1}$, and $x_s=0$, for $s=0,1, \ldots m_j$. Since $C_j=(1+\vp_j)C_{j+1}$, this implies our claim for $i=0$.
   
   Assuming  \eqref{E:1.13.5}  and \eqref{E:1.13.6} are true for $i-1$ with $1\le i < j-1$. Using the recursive definition  of  $y(m_{j-i}, m_{j-i+1},\ldots m_j)$ one  can write it  as
   $$y(m_{j-i}, m_{j-i+1},\ldots, m_j)= \sum_{s=1}^{m_{j-i}} \frac{f(m_{j-i})}{m_{j-i}} e_{j_s} + x_s,$$
   so that  the $x_s$, $s\le m_{j-i}$, are equally distributed vectors, and 
   $\sum_{s=1} ^{m_{j-i}}  x_s$ has the same distribution as $y( m_{j-i+1},\ldots m_j)$. It follows  therefore from the induction hypothesis \eqref{E:1.13.6} (for $i-1$) that 
  $\|x_s\|\le C_{j-i+1}/m_{j-i}$, for $s=1,2\ldots m_{j-1}$. Thus Lemma \ref{L:1.12} is satisfied  with $m=m_{j-i}$, $k=m_{j-i-1}$, $\vp=\vp_{j-i}$ and $C=C_{j-i+1}$,
  and we deduce that  $\|y(m_{j-i}, m_{j-i+1},\ldots, m_j)\le (1+\vp_{j-i}) C_{j-i+1}=C_{j-i}$, which implies \eqref{E:1.13.5}.
  Moreover, the second part of the conclusion of Lemma \ref{L:1.12} yields that
  if we  write  $y(m_{j-i}, m_{j-i+1},\ldots, m_j)$ as sum of  a block of $m_{j-i-1}$ equally distributed vectors $\yt_1<\yt_2< \ldots \yt_{m_{j-i-1}}$, we deduce that
  $\|\yt_t\|\le (1+\vp_{j-i})C_{j-i+1}/m_{j-i-1}=C_{j-i}/m_{j-i-1}$, $t=1,2,\ldots ,m_{j-i-1}$. Since the unit vector basis in $S$ is $1$-unconditional this implies
  that $(y_t: 1\le t\le {m_{j-i-1}})$, with  $y_t=m_{j-i-1}\yt_t/C_{j-i}$, for $t=1,2,\ldots,m_{j-i-1}$, is $C_{j-i}$-equivalent to the $\ell_1^{m_{j-i-1}}$ basis.
  Thus $y(m_{j-i}, m_{j-i+1},\ldots, m_j)/C_{j-i}$, is an $\ell_1^{m_{j-i-1}}$-average up to the constant $1/C_{j-i}$, in the way it is described by 
 \eqref{E:1.13.6}. \end{proof}

  \section{Construction of a version of Gowers Maurey space}\label{S:2}

  To define the space ${\mathcal G \mathcal M}$, which will be a version of the space $GM$ introduced in \cite{GM}, we  need to choose  several  objects.

First, assume that $\vpb=(\vp_n)_{n\ge 0} \subset (0,1)$ satisfies the following {\em standard conditions}
 \begin{equation}\label{E:2.1}  \vp_0< \frac{f(2)-1}{2},  \vp_n\le 2^{-n} \text{ and  }\sum_{i>n}  i^2\vp_{i}\le \frac1{10}  \vp_n,\text{ for $n\kin\N$.}
 \end{equation}

 Secondly, 
 let ${\bf Q}$ be a countable  set   of elements  of $\coo$, so that
 \begin{align}\label{E:2.2}
& \Big\{\sum_{i=1}^l a_ie_i: l\in\N, a_i\in\Q, \text{ for $i=1,2,\ldots,l$}\Big\}\cap[-1,1] ^\N\subset {\bf Q} \subset \coo\cap [-1,1]^\N,\\
\label{E:2.3}
&\text{if $x\in {\bf Q}$ and $E\subset \N$ is finite, then $E(x) \in{\bf Q}$},\\ 
 \label{E:2.4} 
 &\text{if $(x_i)_{i=1}^l\kin {\bf Q}^l$ is a finite block sequence, then $\frac1{f(l)}\sum_{i=1}^l x_i$ and 
    $\frac1{\sqrt{f(l)}}\sum_{i=1}^l x_i$}\\
    &\text{are in $\bf Q$.}\notag
 \end{align}

Next we introduce a  lacunary   set  $J\subset \N$.  We write $J$ as an increasing sequence $\{j_1,j_2,\ldots \}$, and require the following four growth conditions 
\begin{align}
 \label{E:2.5}&\sum_{i>n}\frac{2}{j_i}<\frac{1}{f(j_n)}, \text{ for all $n\in\N$,}\\
 \label{E:2.6}
&\text{$(j_i)_{i=1}^\infty$ is admissible, and satisfies the conditions  \eqref{E:1.13.1} and  \eqref{E:1.13.2} imposed}\\
&\text{on $(m_i)_{j=1}^\infty$ in Lemma \ref{L:1.13} (relative to the sequence $(\vp_n)$ as chosen above).}\notag
\end{align}

In order to formulate the last condition on $J$, we first need to state an observation  which is an easy consequence  of  James' blocking argument.
 \begin{lem} \label{L:2.8}
For all $n\in\N$ and all $\vp>0$ there is an $N=N(n,\vp)$ so that the following holds:

Assume that $(E,\|\cdot\|_E)$ is a Banach space with a normalized and subsymmetric basis $(e_i)$, and  there is a $c\kin(0,1]$  so that for all  $k\kin\N$
$$\Big\|\sum_{j=1}^k e_i\Big\|_E\ge c\frac{k}{f(k)}.$$
Then, for all $\vp>0$ and $n\in\N$,    there is an $m\in[n,N(n,\vp)]$ which is divisible by $n$,  and there are 
$n$ subsets $A_1<A_2<\ldots A_n$ of $\{1,2,\ldots, m\}$, all of cardinality $m/n$,  
so that $(x_i:i=1,2,\ldots, n)$ is $c^{1/n}(1-\vp)$-equivalent to the $\ell_1^n$-unit vector basis, where
$$x_i=\frac{\sum_{j\in A_i} e_i}{\Big\|\sum_{j\in A_i} e_j\Big\|}\text{\ \  for $i=1,2\ldots, n$.}$$
\end{lem}
Our fourth condition  on $J=\{j_1,j_2,\ldots \}$ can now be stated as follows (the first  inequality being trivial):
\begin{align}
\label{E:2.9}
&j_{s}\le N( j_{s},\vp_{s})\le \frac12 \vp_{s+1} f(j_{s+1}).\qquad\qquad
\end{align}
 Finally  we will need an injective function  $\sigma$ from the collection of all
finite sequences of  elements of ${\bf Q}$ to the set $\{j_2,j_4,\ldots\}$ such that
  if $l\in\N$, $z^*_1,z^*_2,\ldots z^*_l\in {\bf Q} $   and $N= \max \big( \cup_{s=1}^i  \supp(z^*_s) \big)$, then
\begin{align}
\label{E:2.10} &\vp_N f(\sigma(z^*_1, \ldots , z^*_i)) \geq N .
\end{align}

 Depending on our choice of $\vp$, $\bf{ Q}$, $J$ and $\sigma$ we  can now define recursively subsets ${\mathcal G \mathcal M}^*_m$ in $c_{00}\cap[-1,1]^{\N}$, for each  $m\in\N_0$, which will
 serve as a set of normalizing functionals of ${\mathcal G \mathcal M}$.
 
 Let 
$$
{\mathcal G \mathcal M}^*_0= \{ \lambda e_n^* : n \in \N, \ | \lambda | \leq 1 \} . 
$$

Assume that ${\mathcal G \mathcal M}_m^*$ has been defined for some $m\in\N_0$. Then ${\mathcal G \mathcal M}_{m+1}^*$ is the set of
all functionals of the form $E(z^*)$ where $E \subseteq \N$ is an interval
and $z^*$ has one of the following three forms   \eqref{E:2.11} , \eqref{E:2.12} or \eqref{E:2.13}:
\begin{equation}\label{E:2.11}
z^*= \sum_{i=1}^l \alpha_i z_i^* 
\end{equation}
where $\sum_{i=1}^l | \alpha_i | \leq 1$ and $z_i^* \in
{\mathcal G \mathcal M}_m^*$  for
$i=1, \ldots , k$.
\begin{equation}\label{E:2.12}
z^*= \frac{1}{f(l)} \sum_{i=1}^l z_i^*, 
\end{equation}
where $z_i^* \in {\mathcal G \mathcal M}_m^*$ for $i=1, \ldots , l$, and
$z_1^*  < \cdots <
z^*_l$. 
\begin{equation}\label{E:2.13}
z^*= \frac{1}{\sqrt{f(k)}} \sum_{i=1}^k z^*_i \mbox{ and }
z^*_i= \frac{1}{f(n_i)} \sum_{j=1}^{n_i} z^*_{i,j}
\end{equation}
where 
\begin{enumerate}
\item[a)]
$z^*_{1,1} < \cdots < z^*_{1,n_1}< z^*_{2,1} < \cdots <
z^*_{l , n_l}$,
\item[b)]   $z^*_{i,j} \in {\mathcal G \mathcal M}^*_m\cap {\bf Q}$, for $1 \leq i \leq k$ and $1 \leq j
\leq n_i$  (and thus $z^*_i\in{\bf Q}$, for $i=1,2\ldots k$), and 
\item[c)]
 $n_1= j_{ 2k'}$,  for some  $k'\ge k$, and $n_{i+1}= \sigma (z_1^*,
\ldots , z_i^*)$, for $i=1, \ldots , k-1$. 
\end{enumerate}
Finally, the norm of ${\mathcal G \mathcal M}$ is defined by
$$
\| x \|_{{\mathcal G \mathcal M}} = \sup \{ z^* (x) : z^* \in \cup_{m=0}^\infty {\mathcal G \mathcal M}^*_m \}.
$$
\begin{remark} There are two  main technical differences between the original space $GM$ defined in \cite{GM} and the space ${\mathcal G \mathcal M}$  defined here:
 \begin{enumerate}
 \item
   we  allow in \eqref{E:2.13} $k$ to take any value in $\N$, while in \cite{GM} $k$ had to be chosen out of the very lacunary set 
  $\{j_{2s+1}$, $ s\in\N\}$ and $\sigma$ in \cite{GM} could only take values in $\{j_{2s}: s\in \N\}$. 
  \item  in \eqref{E:2.13} we allow that $n_1$  is of the form $n_1=j_{2k'}$, with $k'\ge k$, while in \cite{GM}, it is required that $k'=k$.
  \end{enumerate}
The point is that it is not enough to use the coding procedure of \cite{GM} to obtain as they do, given $\epsilon>0$, some $k$ and two intertwined  finite sequences $u_1<v_1<\cdots<u_k<v_k$ such that
$\|\sum_{i=1}^k (u_i-v_i)\| \leq \epsilon \|\sum_{i=1}^k (u_i+v_i)\|$. To deduce  estimates about spreading models, we need this to be valid for any $k$ large enough and for any initial vector $u_1$ far enough along the basis.

  The proof that our construction still
  does not contain an unconditional basis becomes therefore a bit harder. Nevertheless  the main ideas  of the proof stay the same.
 \end{remark}
 
 \noindent{\bf Notation.} 
For $m\in\N$, and if $X$ is a Banach space with a normalized basis $(e_i)$ (we will use this notation for 
 $S$ as well as for ${\mathcal G \mathcal M}$).
$$A_m^*(X)=\Big\{ \frac1{f(l)} \sum_{i=1}^l x_i^*: x^*_1< x^*_2<\ldots x^*_l \text{ in }B_{X^*}
\cap  \coo\Big\}.$$
Note that $A^*_m({\mathcal G \mathcal M})\subset B_{{\mathcal G \mathcal M}^*}$ and $A^*_m(S)\subset B_{S^*}$. 

We define
for $x\in X$  and $m\kin\N$
$$ \|x\|_m=\sup_{x^*\in A^*_m}  |x^*(x)| =\max_{ E_1<E_2<\ldots E_m} \frac1{f(m)}\sum_{i=1}^m \|E_i(x)\|. $$
and observe that
$$\frac1{f(m)} \|x\|_S\le \|x\|_m \le \|x\|_S\le \|x\|_{{\mathcal G \mathcal M}}.$$
For $k\in\N$ we also define

$$\Gamma_k^*=\left\{ \frac1{\sqrt{f(k)}} \sum_{i=1}^k x_i^*: 
\begin{matrix}
x^*_1< x^*_2<\ldots x^*_k \text{ in }B_{{\mathcal G \mathcal M}^*} \cap{\bf Q} \\
x^*_i  \in A^*_{m_i}({\mathcal G \mathcal M}) \text{ with } m_1,m_2,\ldots m_k \in  M\\
m_1=j_{2k'},  k'\ge k,\text{ and } m_{i+1}=\sigma(x^*_1,\ldots x^*_i)\text{ if }i<k 
\end{matrix}
 \right\},$$
 and put for $x\in {\mathcal G \mathcal M}$
 $$ \|x\|_{G^*_m}=\sup_{x^*\in \Gamma^*_m}  |x^*(x)| . $$

 \section{Some technical observations concerning the space ${\mathcal G \mathcal M}$}\label{S:4}
 In this section we prove several properties of the space ${\mathcal G \mathcal M}$, as defined in the previous section. In particular we will conclude 
 that also this version does not contain any unconditional basic sequences. In this section we will abbreviate 
  $\|\cdot\|_{{\mathcal G \mathcal M}}$ by $||\cdot\|$.
 
  The following observation follows from James' blocking argument (See Lemma \ref{L:2.8}).
 \begin{lem}\label{L:4.1} The space  $\ell_1$ is finitely block represented in every infinite dimensional block subspace of ${\mathcal G \mathcal M}$.
 \end{lem} 
 The next Lemma is easy to show (c.f. \cite{Sch1} or \cite{GM})
 \begin{lem}[Action of $\|\cdot\|_l$ on $\ell_1^+$ averages]\label{L:4.2} 
 Assume that $x\in B_{{\mathcal G \mathcal M}}$ is an $\ell_1^{n}$-average  and $l\kin \N$.
 Then 
 \begin{equation}\label{E:4.2.1}
 \|x\|_l\le \frac1{f(l)}\Big(1+\frac{l}{n}\Big).
 \end{equation} 
\end{lem}

\begin{defin}\label{D:4.3} (Rapidly Increasing Sequences)
We call a block sequence  $(x_n)\subset B_{{\mathcal G \mathcal M}}$   {\em rapidly increasing sequence  of constant $c$}  or $c$-RIS,  with $c\in(0,1]$ if the following two conditions \eqref{E:4.3.1} and 
\eqref{E:4.3.2} are satisfied (recall that the sequence $\vp_n$ is given by \eqref{E:2.1}):
\begin{align}
\label{E:4.3.1}   &\text{For $n\in\N$,  $x_n$  is an $\ell_1^{k_n}$-average  of constant $c$, if $c<1$, or of constant}\\
&\text{$1/(1+\vp_n)$, if $c=1$, and the following two inequalities are satisfied:} \notag\\
 &\text{$\max\Big(\frac{2n}{f(k_n)} ,\frac{f(k_n)}{k_n}\Big)< \vp_n^2$ and 
           $f\big(\vp_n\sqrt{k_n}\big)\!\ge\! \frac1{\vp_n^2} \max\supp(x_{n-1})$,  if  $n\ge 2$,} \notag\\
 \label{E:4.3.2}  &\text{$(x_n)$ has a spreading model $E$ 
 with a  1-unconditional and seminormalized} \\
 &\text{basis $(e_i)$ and for  $l\in\N$ and $(a_i)_{i=1}^l\!\subset\!\R$ and $l\kleq n_1\kle n_2<\ldots n_l$ in $\N$} \notag \\
 &\qquad\qquad\frac1{1+\vp_l} \Big\|\sum_{i=1}^l a_i e_i\Big\|_E\le \Big\|\sum_{i=1}^l a_i x_{n_i}\Big\|_E \le (1+\vp_l) \Big\|\sum_{i=1}^l a_i e_i\Big\|_E. \notag 
\end{align}
We say that a sequence $(x_n)$ is an RIS, if it is an $c$-RIS for some constant $c$.  
If $c=1$, we say that $(x_n)$ is an {\em asymptotically isometric RIS}.   
\end{defin}

We note that from Lemma  \ref{L:4.1} it follows immediately that  any  infinite dimensional block subspace
 $Y$   of ${\mathcal G \mathcal M}$ contains  an asymptotically isometric RIS.

 \begin{remark} Let $(x_n)$ be a $c$-RIS, and $(E,\|\cdot\|)$ be the spreading model of $(x_n)$. Define
 for $l\in\N$ 
 \begin{equation}\label{E:4.4}
 g(l)=\frac{l}{\Big\|\sum_{i=1}^l e_i\Big\|_E}=\lim_{n_1<n_2<\ldots n_l}\frac{l}{\Big\|\sum_{i=1}^l x_{n_i} \Big\|}.
 \end{equation}
 From the construction  of ${\mathcal G \mathcal M}$ it follows that 
 \begin{equation}\label{E:4.5}
 g(l)\le f(l)/c \text{ for all $l\kin\N$,}
 \end{equation}
 in particular the spreading model $E$ of $(x_n)$ satisfies the conditions of Lemma  \ref{L:2.8}. 
 
 It follows therefore that  for $n\in\N$ and $\vp>0$ we can
 choose an appropriate $m\in [n, N(n,\vp)]$ and $m$ elements from $(x_j)$ so that  their sum
 is up to a scalar  multiple, which is as  close to $\frac{g(m)}{m}$, as we wish, an $\ell_1^n$ average of constant $c-\vp$. 
 \end{remark}
 Thus it is justified to introduce the following notion of {\em Special Rapidly Increasing Sequences}.

 \begin{defin}\label{D:4.6}   (Special Rapidly Increasing  Sequences)
 A block sequence $(x_n)$ in $B_{GM}$ is called  a {\em Special Rapidly Increasing  Sequence of constant $c$}, with $c\kin(0,1]$,  or $c$-SRIS, if there is a
 $c$-RIS $(\tilde x_n)\subset B_{GM}$, so that for each $n\kin\N$
there is   $\pt=\pt(n)\in P$, $\pt\ge n$,  $\mt(n)\in [j_{\pt(n)}, N(j_{\pt(n)}, \vp_{\pt(n)})] $    (here $N(\cdot,\cdot)$ is chosen as in Lemma \ref{L:2.8}) and natural numbers
  $  \mt(n) \le \st(n,1)<\st(n,2)<\ldots \st(n,\mt(n))$, so that
  \begin{align}
    \label{E:4.6.1}
  &  x_n =  \frac{\tilde g(\mt(n))}{\mt(n)}  \sum_{r=1}^{\mt(n)}  \tilde x_{\st(n,r)}, \\
   \intertext{where $$\tilde g(l)=\frac{l}{\Big\| \sum_{i=1}^l \tilde e_i\Big\|_{\tilde E}}= \lim_{n_1<n_2<\ldots n_l} \frac{l}{\Big\| \sum_{i=1}^l \tilde x_{n_i}\Big\|_{\tilde E} }$$
  and $\tilde E$ is the spreading model of  $(\tilde x_n)$ with semi normalized $1$-unconditional basis   $(\tilde e_n)$,}
\label{E:4.6.2}  & x_n \text{ is an $\ell_1^{j_{\pt(n)} }$ -average of constant $c$, if $c<1$, or 
$1/(1+\vp_{\pt(n)})$, if $c=1$,}\\
\label{E:4.6.2bis}  & (x_n) \text{\ is an RIS of constant $c$, with $k_n=j_{\pt(n)}$ in condition \eqref{E:4.3.1}.} 
  \end{align}
  \end{defin}
\begin{remark}  From the remark before Definition \ref{D:4.6}  it follows that every block subspace contains special rapidly increasing sequences.
The point  of Definition \ref{D:4.6}  is that we may regard the $x_n$ at the same time  as $\ell_1^{k_n}$ -averages for fast increasing $k_n$, but
 also, up to some factor, sums of elements of an RIS. We shall use this to prove that SRIS generate spreading models equivalent to the unit vector basis of $S$. 
 Note that every normalized  block basis of $\mathcal G\mathcal M$ dominates the unit basis of $S$. 
 \end{remark}

\begin{lem}[Action of $A^*_l$ on sums of elements of an RIS] \label{L:4.7} Assume that $(x_n)$ is a  $c$-RIS, $c\in(0,1]$.  Let   $m\le n_1<n_2<\ldots n_m$ be in $\N$ and $(a_i)_{i=1}^m\in \R^m$.
 Put $y=\sum_{i=1}^m  a_ix_{n_i}$.
\begin{enumerate}
\item[a)]If $f(l)\le 2m/\vp_{n_1}$  then there are numbers $l_1$ and $l_2$ in $\N$, with $l_1+l_2\le \min(2l,m)$, 
intervals $I_1<I_2<\ldots< I_{l_1}$ in $\{1,2\ldots m\}$, so that $l_2=\#I_0$, with $I_0=\{1,2\ldots \}\setminus \bigcup_{j=1}^{l_1}I_j$ and
\begin{equation}\label{E:4.7.1} 
\|y\|_\ell\le \frac1{f(l)}\left[ \sum_{j=1}^{l_1} \Big\|\sum_{s\in I_j} a_s x_{n_s}\Big\|+
 \sum_{s\in I_0 } |a_s| \frac{1+2\vp_{n_s}}{c} \|x_{n_s}\|\right],
\end{equation}
\item[b)] If $f(l)> 2 m/\vp$,  for some $\vp\in[\vp_{n_1},1)$, then
\begin{equation}\label{E:4.7.2} \|y\|_l \le   \max_{i\le m}|a_i|\big[ 2\vp + \max_{i=1,2\ldots  m} \|x_{n_i}\|_l\big]\le   \max_{i\le m}|a_i|\big[  2\vp +1].
\end{equation}
\end{enumerate}
\end{lem}

\begin{proof} We put $z_s=x_{n_s}$ for $s=1,2\ldots m$.  In order to prove (a) we choose finite intervals
$E_1<E_2<\ldots E_l$  of $\N$, so that 
$$\|y\|_l=\frac1{f(l)} \sum_{t=1}^l \|E_t(y)\|.$$
Without loss of generality we can assume that
$$\bigcup_{t=1}^l E_t=\ran(y)=[\min\supp(z_1),\max\supp(z_{m})].$$
For $t=1,2\ldots l$  we divide $E_t$ in three intervals $E^{(1)}_t$,  $E^{(2)}_t$,  $E^{(3)}_t$ (some of them possibly empty)
as follows: we let, if it exists, 
 $m(t)$ be the unique number  in $\{1,2,\ldots m\}$    so that  $\min\ran(z_{m(t)} )< \min E_t\le \max\ran(z_{(m(t)} )$ and put
$$E^{(1)}_t= E_t\cap[1,\max \ran(z_{m(t)})].$$
If $m(t)$ does not exists we let $E^{(1)}_t=\emptyset$.
Then we let $m'(t)$ be the unique number $m'(t)$, if it exists, 
so that   $\min\ran(z_{m'(t)} )< \max E_t\le \max\ran(z_{(m(t)} )$ ,
and put 
 $$E^{(3)}_t=\big(E_t\cap [\min\ran(z_{m'(t)}),\infty) \big)\setminus E^{(1)}_t.$$
If $m'(t)$ does not exists we let $E^{(1)}_t=\emptyset$.
 Finally we let $E^{(2)}_t\setminus (E^{(1)}_t\cup E^{(3)}_t)$.
 Let $\tilde \cE$ be the non empty elements of $\{E_t^{(1)},E_t^{(2)},E_t^{(3)}:t\le l\}$ and $\tilde l$ the cardinality of $\tilde\cE$.
 
 We note that  
 $\tilde \cE$ consists of pairwise  disjoint  intervals which can be ordered into  $\Et_1<\Et_2<\ldots \Et_{\lt}$,
  and that for any $i\le m$ and any $j\le \lt$, either $\Et_j$ contains $\ran(z_i)$, or is contained in $\ran(z_i)$, or 
   $\Et_j$ and $\ran(z_i)$ are disjoint.

 For $s\in\{1,2,\ldots,m\}$ we  deduce from our condition on $f(l)$ and  \eqref{E:4.3.1} that
  $$\frac{l}{k_{n_s}}\le \frac{f(l)}{f(k_{n_s})}\le \frac{2m}{\vp_{n_1}f(k_{n_s})} \le\frac{2n_s}{\vp_{n_1}f(k_{n_s})} 
 \le  \frac{\vp_{n_s}^2}{\vp_{n_1}}\le \vp_{n_s}.$$
 We let  
 $$I_0=\big\{ s=1,2\ldots,m : \#\{t:\Et_t\subset \ran(x_s)\}\ge 2 \big\},$$
  
  Lemma \ref{L:4.2} yields that for  every $s\in I_0$
$$\sum_{t, \Et_t \subset \ran(z_s)} \|\Et_t(z_s)\|\le 1+\frac{l}{k_{n_s}}
 \le 1+{\vp_{n_s}}\le \frac{\|z_{s}\|+2\vp_{n_s}}c. $$
We reorder the set $\tilde\cE'$ of all sets $\Et_t$,  $t\in \{1,2\ldots \lt\} $,  which contain  the range of  at least one $x_{n_s}$, 
  into $E'_1<\ldots E'_{l_1}$ and
 we define $t=1,\ldots l_1$
 $$I_t=\big\{s\in \{1,2\ldots m\}: \ran(x_s)\subset \Et_t\big\},$$
 and conclude that $l_1+l_2\le \min(m,2l)$, where  $l_2=\#I_0$, and
 \begin{align*}
\frac1{f(l)}\sum_{t=1}^l \|E_i(y)\|&\le \frac1{f(l)}\sum_{t=1}^{\lt} \|\Et_t(y)\|\\
 &=  \frac1{f(l)}\left[ \sum_{t=1}\Big\|\sum_{s\in I_t} a_s z_s\Big\|  +\sum_{s\in I_0}|a_s|\sum_{t=1}^{\lt}\|\Et_t(z_s) \| \right]\\
 &\le  \frac1{f(l)}\left[ \sum_{t=1}\Big\|\sum_{s\in I_t} a_s z_s\Big\|  +\sum_{s\in I_0}|a_s| \frac{\|z_s\|+\vp_{n_s}}c \right]
 \end{align*}
   which implies (a).

In order to prove our claim  (b) let $\vp\in[\vp_{n_1},1]$ and  define 
$$i_0=\max\big\{ i=1,2\ldots k:   \max(\supp( z_{i-1})) <f(l) \vp \big\}$$
(with  $\max(\supp( z_{0})) :=0$).

Then by \eqref{E:4.3.1}  it follows for  $i\in \{i_0+1, i_0+2, \ldots, m\}$   that
$$f(k_{n_i})>\frac1{\vp_{n_i}}\max\supp(z_{{i-1}})\ge    \frac1{\vp}\max\supp(z_{i_0}) \ge       f(l)  $$
and thus, 
\begin{align*}
\Big\|\sum_{i=1}^m a_i z_{i}\Big\|_l&\le \frac1{f(l)} \sum_{i=1}^{i_0-1}|a_i| \|z_i\|_{l} +|a_{i_0}|\|z_{i_0}\|_l+\frac1{f(l)}\sum_{i=i_0+1}^m |a_i|\|z_{i}\|_l\\
 &\le \max_{i\le m} |a_i| \Bigg[ \frac{\max\supp(z_{i_0-1})}{f(l)}+\|z_{i_0}\|_l+\frac1{f(l)} \sum_{i=i_0+1}^m\Big(1+\frac{l}{k_{n_i}}\Big)\Bigg]\\
  &\quad\text{ (by  Lemma \ref{L:4.2})}\\
 &\le  \max_{i\le m} |a_i| \Big[   \vp+\|z_{i_0}\|_l+\frac{2m}{f(l)} \Big]<  \max_{i\le m} |a_i| \big[ 2\vp+ \|z_{i_0}\|_l \Big]
\end{align*}
which proves part (b).
\end{proof}

\begin{lem} [Action of $\Gamma^*_k$ on sums of elements of an RIS] \label{L:4.7bis} Assume that $(x_n)$ is a  block-sequence in ${\mathcal G \mathcal M}$, $c\in (0,1]$, and let $z^* \in \Gamma^*_k$. Let   $m\le n_1<n_2<\ldots n_m$ be in $\N$ and $(a_i)_{i=1}^m\in \R^m$.
 Put $y=\sum_{i=1}^m  a_ix_{n_i}$.  
\begin{enumerate}
\item[a)]If $(x_n)$ is a $c$-RIS, then
 \begin{align}\label{E:4.7.89}
|z^*(y)|
&\le  \frac{|z^*_{t_0}(y)|}{\sqrt{f(k)}} +\frac{\max_{s\le m}|a_s|}{\sqrt{f(k)}}\left[1+ 2m\vp_{n_1}+\sum_{t\in T_0} \max_{s\in S_t} \|x_{n_s}\|_{l_t}     \right]\\
 &\qquad+ \frac1{\sqrt{f(k)}} \sum_{s=s_0+1}^m |a_s| \Big|  \sum_{t\in T_s} z^*_t(x_{n_s})\Big|,\notag
\end{align}
where $t_0\in\{1,2\ldots k\}$, $T_s, \subset \{t_0+1,t_0+2\ldots k\}$, $s=0,1,2\ldots m$ are defined as follows:
\begin{align*} &t_0=\min\{ t=1,\ldots k:                 z^*_t(y)\not=0\},\\
&\text{(we assume that $t_0$ exists, otherwise $z^*(y)=0$)}\\
&T_s=\big\{t\!=\!t_0\!+\!1,\ldots k: \supp(z^*_t)\!\subset\! [\min\ran(z_{s}), \min \ran(z_{s+1}))\big\}, \text{ if  $s\!=\!1,\ldots k$ }\\
&\text{(with $\min\ran(z_{m+1}):=\infty$)} \\
&T_0=\{ t_0+1,\ldots k\} \setminus \bigcup_{s=1}^k T_s.
\end{align*}
\item[b)] If $(x_n)$ s a $c$-SRIS, then
\begin{equation}\label{E:50} |z^*(y)| \le
\frac{|z^*_{t_0}(y)|}{\sqrt{f(k)}} +\frac{\max_{s\le m}|a_s| }{\sqrt{f(k)}}\min(m,k) \max_{s\le m, t_0<t\le k} \|x_s\|_{l_t} + 2\max_{s\le m}|a_s|.
\end{equation}
\end{enumerate}

\end{lem}

\begin{proof}
We first assume that $(x_n)$ is  only a $c$-RIS . For $m\le n_1<n_2<\ldots n_s$ in $\N$  and $(a_s)_{s=1}^m\subset \R\setminus\{0\}$ we put
$$y=\sum_{s=1}^m a_s x_{n_s}.$$
Secondly let $k\in\N$ and $z^*\in \Gamma^*_k$. 
We write $z^*$ as
$$z^*=\frac1{\sqrt{f(k)}}\sum_{t=1}^k z^*_t \in \Gamma^*_k,$$
with $z^*_1\in A^*_{l_1}$, and  $l_1=j_{2k'}$, for some $k'\ge k$, and $z^*_{i}\in  A^*_{l_{i}}$, with $l_i=\sigma(z^*_1,z^*_2,\ldots,z^*_{i-1})$, for $i=2,\ldots k$.
We note that for $t_0$ as defined in the statement it follows that
$$\max(\supp(z^*_{t_0}))\ge \min( \supp(x_{n_1})).$$ 
We abbreviate $z_s=a_s x_{n_s}$, for $s=1,2\ldots m$.
For $t\kin T_0$ let 
$$S_t=\big\{ s\in\{1,2\ldots m\}: \ran(z_t^*)\cap \ran(z_s)\not=\emptyset\big\}.$$
Note that $S_t$ is an interval in $\{1,2,\ldots m\}$ and that
each $s\in\{1,2\ldots,m\}$ maybe element in at most two of the sets $S_t$, $t\kin T_0$.
Using these notations, we can now write $z^*(y)$ as
\begin{align}\label{E:4.8}z^*(y)&=\frac{1}{\sqrt{f(k)}}z^*_{t_0}(y) +\frac{1}{\sqrt{f(k)}}\sum_{t\in T_0} \sum_{s\in S_t} z^*_t(z_s)
 +\frac{1}{\sqrt{f(k)}} \sum_{s=1}^m \sum_{t\in T_s} z^*_t(z_s).
\end{align}
  In order to estimate the second term in \eqref{E:4.8} we first deduce from
  \eqref{E:2.10}  and the trivial estimate $\min\supp(z_1)> 2n_1$, that for $t\kin T_0$
  $$f(l_t)\ge \frac{\max\supp(z^*_{t_0})}{\vp_{\max\supp(z^*_{t_0})} }\ge \frac{\min\supp(z_1)}{\vp_{n_1}}>\frac{2n_1}{\vp_{n_1}}\ge \frac{2m}{\vp_{n_1}},$$
  and Lemma \ref{L:4.7} (b) yields therefore that
  \begin{equation}\label{E:4.9}
  \Big|z^*\Big(\sum_{s\in S_t} z_s \Big)\Big|\le \max_{s\in S_t}|a_i| \big[ 2\vp_{n_1}+\max_{s\in S_t}\|x_{n_s}\|_{l_t}\big], \text{ whenever $t\in T_0$.}
  \end{equation}
  
  In order to estimate the third term in \eqref{E:4.8} we define
  \begin{equation}\label{E:4.10} s_0= \min\big\{s=1,2\ldots m: \max \supp(z_{s-1})<\vp_{n_1} \sqrt{f(k)}\big\}\end{equation}
with the usual convention that  $\max\supp(z_{0}) =0$. 
We first note that if $s_0\ge 2$
\begin{equation}\label{E:4.11}
\frac1{\sqrt{f(k)}} \sum_{s=1 }^{s_0-1}  \sum_{t\in T_s} |z^*_t(z_s)|\le
\frac1{\sqrt{f(k)}}\max\supp(z_{s_0-1})\max_{s\le m}|a_s|\le \vp_{n_1}\max_{s\le m}|a_s|.
\end{equation}
Secondly, if $T_{s_0}\not=\emptyset$ we  let 
$$I=\Big[\min\bigcup_{t\in T_{s_0}} \ran(z^*_t),\max\bigcup_{t\in T_{s_0}} \ran(z^*_t)\Big],$$
and deduce that
\begin{equation}\label{E:4.12}
\frac1{\sqrt{f(k)}} \Big| \sum_{t\in T_{s_0}}  z^*_t(z_{s_0})\Big| =|I(z^*)(z_{s_0})|\le |a_{s_0}|\cdot\|x_{n_{s_0}}\| \le \max_{s\le m}|a_s|.
\end{equation}
Adding up the estimates  obtained in   \eqref{E:4.11} and \eqref{E:4.12} and inserting them into \eqref{E:4.8}, we obtain
\begin{align}\label{E:4.13}
|z^*(y)|&\le \frac{|z^*_{t_0}(y)|}{\sqrt{f(k)}}+\frac1{\sqrt{f(k)}}\sum_{t\in T_0}  \max_{s\in S_t}|a_s|\big[\max_{s\in S_t} \|x_{n_s}\|_{l_t} +2\vp_{n_1}\big] \\
&\qquad +\max_{s\le m}|a_s| + \frac1{\sqrt{f(k)}} \sum_{s=s_0+1}^m |a_s| \Big|  \sum_{t\in T_s} z^*_t(x_{n_s})\Big| \notag\\
&\le  \frac{|z^*_{t_0}(y)|}{\sqrt{f(k)}} +\frac{\max_{s\le m}|a_s|}{\sqrt{f(k)}}\left[1+ 2m\vp_{n_1}+\sum_{t\in T_0} \max_{s\in S_t} \|x_{n_s}\|_{l_t}     \right] \notag\\
  &\qquad + \frac1{\sqrt{f(k)}} \sum_{s=s_0+1}^m |a_s| \Big|  \sum_{t\in T_s} z^*_t(x_{n_s})\Big|,\notag
\notag\end{align}
which proves \eqref{E:4.7.89}.

In order to prove part (b) we now assume that $(x_n)$ is an SRIS of constant $c$,
 and want for  $s=s_0+1,\ldots m$, with $T_s\not=\emptyset$, to  find an upper estimate for  
$$ \frac1{\sqrt{f(k)}}
 \Big|  \sum_{t\in T_s} z^*_t(x_{n_s})\Big|.$$ 
 Thus, we assume that there is an RIS $(\xt_n)\subset B_{{\mathcal G \mathcal M}}$ of constant $c$,  and for each $n\in\N$ numbers
 $\pt(n)\in\N$, $\pt(n)\ge n$, $\mt(n)\in [j_{\pt(n)}, N(j_{\pt(n)},\vp_{\pt(n)}]$ and $\mt(n)\le \st_1(n)<s_2(n)<\ldots \st_{\mt(n)}(n)$ so that
  $$x_{n}=\frac{\gt(\mt(n))}{\mt(n)} \sum_{r=1}^{\mt(n)} \xt_{\st_r(n)},$$
 where $\gt:[1,\infty)\to [1,\infty)$ is an increasing function with $\gt(\xi)\le f(\xi)/c$, $\xi\ge 1$ and so that
 $k_{n}=j_{\pt(n)}$ (thus $x_{n}$ is an $\ell_1^{\pt(n)}$-average of constant $c$).
 
 We fix  $s=s_0+1,s_+2,\ldots m$, with $T_s\not=\emptyset$, and apply now our estimate \eqref{E:4.13}  to 
 $x_{n_s}=\frac{\gt(\mt(n_s))}{\mt(n_s)} \sum_{r=1}^{\mt(n_s)} \xt_{\st_r(n_s)}$ instead of $y$ and to $\zt^*=\frac1{\sqrt{f(k)}}\sum_{t\in T_s} z^*_t$, instead of $z^*$.
 Strictly speaking $\zt^*$ is not in $\Gamma_k^*$, but it is of the form $I(z^*)$, where $I\subset \N$ is an interval, and it is easy to see that 
 it satisfies the same estimates \eqref{E:4.13}.   From the definition of $s_0$ and by the second condition on $k_{n_s}$ in \eqref{E:4.3.2} we deduce that
  $$\sqrt{f(k)}< \frac1{\vp_{n_1}} \supp(z_{s-1}) <\sqrt{f(\vp_{n_s}k_{n_s}^{1/2})},$$
  and thus 
  $$k\le \vp_{n_s} k_{n_s}^{1/2} =\vp_{n_s}j_{\pt(n_s)}^{1/2}\le\vp_{n_s} \mt^{1/2}_s(n_s),$$
  which yields 
  $$\frac{k \gt(\mt(n_s))}{\mt(n_s)} \le \vp_{n_s} \frac{\gt(\mt(n_s))}{\mt^{1/2}(n_s)}\le  \vp_{n_s}.$$
  Put  $t_s=\min T_s$ (which takes the role of $t_0$).
   If we apply \eqref{E:4.13} to $x_{n_s}$, the sets $T_0$ and $T_s$ will be replaced by sets
   $\tilde T_0$ and $\tilde T_{\st}$, $\st\le \mt(n_s)$, for which we know (and only will need to know) that
   $$\#\tilde T_0\le k\text{ and }\sum_{\st=1}^{\mt(n_s)}\#\tilde T_{\st}\le k.$$
  Using also the estimate
  $m\vp_{n_1} +\|\xt_{\st_r(n_s)}\| +1\le 3$, we obtain from   \eqref{E:4.13}, applied to $\zt^*(x_{n_s})$, that 
 \begin{align*}
   |\zt^*(x_{n_s})|&\le \frac{|z^*_{t_s}(x_{n_s}|}{\sqrt{f(k)}} +\frac{\gt(\mt)}{\mt}\frac{3k}{\sqrt{f(k)}}+\frac{\gt(\mt)}{\mt}\frac{k}{\sqrt{f(k)}}
                                 \le  \frac{|z^*_{t_s}(x_{n_s})|}{\sqrt{f(k)}} +4\vp_{n_s}.
\end{align*}
Inserting this estimate back  into  \eqref{E:4.13} for $|z^*(y)|$ we get
\begin{align}\label{E:4.14}
  |\zt^*(y)|&\le\frac{|z^*_{t_0}(y)|}{\sqrt{f(k)}} +\frac{\max_{s\le m}|a_s|}{\sqrt{f(k)}}\left[1+ 2m\vp_{n_1}  +\sum_{t\in T_0} \max_{s\in S_t} \|x_{n_s}\|_{l_t}    \right]\\
  & \qquad+\frac1{\sqrt{f(k)}} \sum_{s=1, T_s\not=\emptyset}^m  |a_s|\big[|z^*_{t_s}(x_{n_s})  + 4\vp_{n_s}\big]\notag\\
  &\le\frac{|z^*_{t_0}(y)|}{\sqrt{f(k)}} +\frac{\max_{s\le m}|a_s|}{\sqrt{f(k)}}\left[\sum_{t\in T_0} \max_{s\in S_t} \|x_{n_s}\|_{l_t} + \sum_{s=1, T_s\not=\emptyset}^m \|x_{n_s}\|_{l_{t_s}}
  +2   \right]\notag\\
  &\le \frac{|z^*_{t_0}(y)|}{\sqrt{f(k)}} +\frac{\max_{s\le m}|a_s| }{\sqrt{f(k)}}\min(m,k) \max_{s\le m, t_0<t\le k} \|x_s\|_{l_t} + 2\max_{s\le m}|a_s| ,\notag
\end{align}
which proves our claim (b).
\end{proof}
  
  \begin{lem}\label{L:4.15}
 Assume that $(x_n)$ is an SRIS of constant $c$, $c>0$ and assume $z^*\in \Gamma^*$.
  As before write $z^*$ as
$$z^*=\frac1{\sqrt{f(k)}}\sum_{t=1}^k z^*_t \in \Gamma^*_k,$$
with $z^*_1\in A^*_{l_1}$, and  $l_1=j_{2k'}$, for some $k'\ge k$, and $z^*_{i}\in  A^*_{l_{i}}$, with $l_i=\sigma(z^*_1,z^*_2,\ldots,z^*_{i-1})$, for $i=2,\ldots k$
  and assume that 
$$ t_0=\min\{ t=1,\ldots k:  z^*_t(y)\not=0\},$$
 exists (otherwise  $z^*(y)=0$).
 Let $m$ and $m\le n_1<n_2< \ldots n_m$ be in $\N$, $(a_s)_{s=1}^m\subset \R$ and assume that the numbers $j_{\pt(n_s)}$  (as chosen in Definition \ref{D:4.6})
 are all different from the numbers $l_{t_0+1}$,  $l_{t_0+2}$, ....$l_{k}$.
 
 Then it follows that
 \begin{align}\label{E:4.15.1}
  |z^*_t(x_{n_s})|&\le \vp_{n_1}, \text{ for $t=t_0+1,t_0+2,\ldots k$ and $s=1,2\ldots m$, and }\\
 \label{E:4.15.2} 
   |z^*(y)|&\le \frac{|z^*_{t_0}(y)|}{\sqrt{f(k)}}+ 3\max_{s\le m}|a_s|, \text{ where $y=\sum_{s=1}^m a_s x_{n_s}$}.
 \end{align}
  \end{lem}
  \begin{remark} The assumption of Lemma \ref{L:4.15} are for example satisfied if in Definition \ref{D:4.6} the numbers $\pt(n)$, $n\kin\N$, are chosen to be  odd numbers
  (since the image of $\sigma$ is a subset of $\{j_{2i}:i\kin\N\}$).  
  \end{remark}
  \begin{proof}[Proof of Lemma \ref{L:4.15}]
  We need  to estimate   $\|x_{n_s}\|_{l_t} $ for $s\in \{1,2\ldots \}$ and $t=t_0+1,t_0+2,\ldots k$ and then apply \eqref{E:50}. Recall that 
    \begin{equation}\label{E:4.15.3}
    x_{n_s}=\frac{\gt(\mt(n_s))}{\mt(n_s)} \sum_{r=1}^{\mt(n_s)} \xt_{\st_r(n_s)},
    \end{equation}
 where $\gt:[1,\infty)\to [1,\infty)$ is an increasing function with $\gt(\xi)\le f(\xi)$, $\xi\ge 1$ and so that
 $k_{n_s}=j_{\pt(n_s)}$,  $(\xt_n)$ is an RIS of constant $c$, and $\mt\in[j_{\pt(n_s)},  N(j_{\pt(n_s)}, \vp_{\pt(n_s) })]$, and 
 $\mt(n_s)\le \st_1(n_s)<\ldots <\st_{\mt}(n_s)$.
 
   We note that either $l_t<k_{n_s}=j_{\pt(n_s)}$, then, since $z_s$  is an  $\ell_1^{j_{\pt(n_s)}}$-average, we deduce from Lemma \ref{L:4.2}, and 
 \eqref{E:2.10} 
 that 
 \begin{equation}\label{E:4.15.4}
  \|x_{n_s} \|_{l_{t}}\le \frac{2}{f(l_{t})} \le\frac2{f( \sigma(z^*_1,z^*_2,\ldots z^*_{t-1}))}\le \vp_{n_1}.
  \end{equation}
 Or   we have that $l_{t}>k_{n_s}=j_{\pt(n_s)}$.   This implies  by \eqref{E:2.9} that
  $$f(l_t)\ge f(j_{\pt(n_s)+1})\ge \frac{2}{\vp_{\pt(n_s)+1}} N(j_{\pt(n_s)}, \vp_{\pt(n_s) })\ge \frac{2\mt(n_s)}{\vp_{\pt(n_s)+1}}\ge \frac{2\mt(n_s)}{\vp_{n_1}}.$$
 But then it follows from \eqref{E:4.15.3},
  Lemma \ref{L:4.7} (b),  and \eqref{E:4.3.1} that
    \begin{equation}
 \|x_{n_s}\|_{l_{t}}  \le 2\frac{\tilde g(\mt(n_s))}{\mt(n_s) }\le 2\frac{{f (\mt(n_s))}}{\mt(n_s) }    \le 2\frac{f (k_{n_s})}{k_{n_s}  }  \le  \vp_{n_1}.
   \end{equation}
Thus, \eqref{E:50} yields
  \begin{align*}
  |\zt^*(x_{n_s})|&\le \frac{|z^*_{t_0}(y)|}{\sqrt{f(k)}} +\frac{\max_{s\le m}|a_s| }{\sqrt{f(k)}}\min(m,k) \vp_{n_1} + 2\max_{s\le m}|a_s| 
  &\le \frac{|z^*_{t_0}(y)|}{\sqrt{f(k)}}+ 3\max_{s\le m}|a_s|,
  \end{align*}
  which proves our claim.
 \end{proof} 
 
We now can formulate  and prove our Key Lemma.

\begin{lem}\label{L:4.16}  For each $c\in(0,1]$ there is a constant $C=C_c>0$ so that the following holds. 

Let $(x_n)$ be a $c$-SRIS, 
and assume that  the $\pt(n)$, $n\in\N$, as in Definition \ref{D:4.6} are chosen to be odd numbers.  
 Let $m\le n_1<n_2<\ldots n_m$ be in $\N$ and put
  $y=\sum_{s=1}^m x_{n_s}$. 
Then
\begin{enumerate}
\item[a)]   $c \frac{m}{f(m)}\le   \|y\|\le  C\frac{m}{f(m)}$, if $c<1$  and 
                    $(1-\vp_{n_1}) \frac{m}{f(m)}\le   \|y\|\le  C\frac{m}{f(m)}$, if $c=1$.
\item[b)]      For $l\in \N$, $l\ge 2$ 
    $$\|y\|_l\le \begin{cases} \frac{C}{f(l)} \frac{m}{f(m/\min(2l,m))}   &\text{ if $f(l)\le m/\vp_{n_1}$,}\\
                                         2 C    &\text{ if $f(l)\ge m/\vp_{n_1}$.}
                    \end{cases}$$
\end{enumerate}
\end{lem}
\begin{remark} Of course we can (and will later) replace $2C$ in the second case of (b) in Lemma \ref{L:4.16} by an another constant. Nevertheless the ``$2C$'' is needed so that 
the induction argument in the proof will work out.
\end{remark}
\begin{proof} The first inequality in (a) follows from the fact that
$$\|y\|\ge \|y\|_m \ge \frac1{f(m)}\sum_{i=1}^m \|x_{n_s}\|\ge c\frac{m}{f(m)},$$
if $c<1$. A similar argument works for $c=1$.

 Using the first condition in \eqref{E:2.1}   it is easy  to see that one can choose  $m_0\in\N$ so that
 \begin{align}\label{E:4.16.1}
  f(m)&\le \frac{f(l) f(m/2l)}{1+\vp_0},\,\, 2m\vp_{m}<c\vp_{m-1} \text{ and }\frac{f(m)}{f(m/4)} \le 2,\\
  &\qquad\qquad\qquad\qquad\text{ whenever $m\ge m_0$ and $2\le l\le m/4$}\notag
  \end{align}
 (note that the second inequality  is satisfied  as long as $c\ge1/m $, by the third condition in \eqref{E:2.1}).
Put $C=4 m_0$. 

We will prove the second inequality  in  (a) and (b)  by induction for each $m\in\N$. If $m\le m_0$ (a) and (b)  are trivial.
 
 So assume (a) and (b) are true for all $m'<m$, for some $m\ge m_0$. Let
  $m\le n_1< n_2< \ldots<n_m$ be in $\N$ and put
  $$y=\sum_{s=1}^m x_{n_s}.$$

 For $l\in \N$, $l\ge 2$, we first estimate $\|y\|_l$.
 
  If $f(l)\ge m/\vp_{n_1}$ then Lemma \ref{L:4.7} (b) implies that 
 $\|y\|_l\le 2\le  C$. 
 
  If $f(l)\le m/\vp_{n_1}$ it follows from  the second part of \eqref{E:4.16.1} and
 Lemma \ref{L:4.7} (a) that there are  natural numbers $0=s_0<s_1<\ldots s_{l'}=m$,
 with $l'=\min(2l,m)$
  so that 
\begin{align*}
 \|y\|_l&\le\frac{ \vp_{n_1-1}}{f(l)}   +\frac1{f(l)}  \sum_{j=1}^{l'}  \Big\|\sum_{s=s_{j-1}+1}^{s_j}  x_{n_s} \Big\|.\\
 \end{align*}
 If $l\ge m/4$ then  by   the third part of \eqref{E:4.16.1} 
 $$\|y\|_l \le\frac{ \vp_{n_1-1}}{f(l)} + \frac{m}{f(l)}\le \frac{ \vp_{n_1-1}}{f(l)}     +\frac{m}{f(m/4)}\le \frac{ \vp_{n_1-1}}{f(l)} +2 \frac{m}{f(m)}\le C\frac{m}{f(m)}.$$
 If $l\le m/4$ we are using the induction hypothesis and the fact that the map $[1,\infty)\ni x\mapsto x/f(x)$ is concave to obtain 
 \begin{align*}
 \|y\|_l&\le\frac{ \vp_{n_1-1}}{f(l)}   +\frac1{f(l)}  \sum_{j=1}^{l'} \Big\| \sum_{s=s_{j-1}+1}^{s_j} x_{n_s}\Big\|\\
& \le\frac{ \vp_{n_1-1}}{f(l)}   +\frac{C}{f(l)}  \sum_{j=1}^{l'}    \frac{s_{j}-s_{j-1}}{f(s_{j}-s_{j-1})}  \\
&\le \frac{ \vp_{n_1-1}}{f(l)}   +\frac{C}{f(l)} l' \frac{m/l'}{f(m/l')}\\
&\le \frac{ \vp_0}{f(2)}   +\frac{C}{1+\vp_0}   \frac{m}{f(m)} \text{ (by first condition in \eqref{E:4.16.1})}\\
&\le C \frac{m}{f(m)} ,
 \end{align*}
for the last inequality note that
$$\frac{\vp_0}{f(2)} \le C\frac{\vp_0}{1+\vp_0} \frac2{f(2)}= C\Big[1-\frac1{1+\vp_0}\Big] \frac{2}{f(2)} \le C\Big[1-\frac1{1+\vp_0}\Big] \frac{m}{f(m)}.$$

This proves that  $\|y\|_l\le Cm/f(m)$, for every $l\ge 2$. Together with Lemma \ref{L:4.15} 
  (which estimates $\|y\|_{\Gamma^*_k}$ for $k\in\N$)
this yields that $\|y\|\le  Cm/f(m)$. That finishes the induction step and the proof of (a).

Part (b) follows  if $f(l)\ge m/\vp_{n_1}$ directly from Lemma \ref{L:4.7} (b). If  $f(l)\le m/\vp_{n_1}$ we apply Lemma \ref{L:4.7} (a),
the concavity of  the map $[1,\infty)\ni x\mapsto x/f(x)$ and part (a) of this lemma, to obtain
 for some choice of  natural numbers $0=s_0<s_1<\ldots s_{l'}=m$,
 with $l'=\min(2l,m)$
  so that 
\begin{align*}
 \|y\|_l&\le\frac{ \vp_{n_1-1}}{f(l)}   +\frac1{f(l)}  \sum_{j=1}^{l'} \Big\|\sum_{s=s_{j-1}+1}^{s_j}  x_{n_s} \Big\|\\
         &  \le\frac{ \vp_{n_1-1}}{f(l)}   +\frac{C}{f(l)}  \sum_{j=1}^{l'}    \frac{s_{j}-s_{j-1}}{f(s_{j}-s_{j-1})}  \\
  &\le \frac{ \vp_{n_1-1}}{f(l)}   +\frac{C}{f(l)} \frac{m}{f(m/l')}\le\frac{2C}{f(l)} \frac{m}{f(m/l') }, \end{align*}
which proves our claim.
\end{proof}
\begin{remark}
Following now the proof in \cite{GM}  (from Lemma 7 in \cite{GM}    on)   one deduces that ${\mathcal G \mathcal M}$, as defined here has also no 
 unconditional basis sequence. In Section \ref{S:6} (see Theorem \ref{T:6.14}) we will prove that in every block subspace of ${\mathcal G \mathcal M}$ there are two   seminormalized block sequence 
$(u_n)$ and $(v_n)$, which are   {\em intertwinned}, i.e. $u_1<v_1<u_2<v_2<\ldots$, with the property that for some constants $0<c,C<\infty$
\begin{align*}
&\Big\|\sum_{s=1}^l u_{n_s} +v_{n_s}\Big\|\ge c \frac{l}{\sqrt{f(l)}} \text{ and }
\Big\|\sum_{s=1}^l u_{n_s} -v_{n_s}\Big\|\le C \frac{l}{f(l)} .
\end{align*}
for all $l\in\N$ and all choices of  $l\le n_1<n_2<\ldots n_l$ in $\N$. This certainly implies that ${\mathcal G \mathcal M}$ has no unconditional block sequence.

 We do not know whether or not ${\mathcal G \mathcal M}$ is HI, but we suspect it is. The point is that to use spreading models and other refinements, we needed to pass to subsequences of the Rapidly Increasing Sequences as defined in $GM$. Therefore we lost the freedom to pick the 
vectors of an RIS-sequence in arbitrary subspaces, as would be needed to repeat Gowers-Maurey's proof that $GM$ is HI.

 Nevertheless, Gowers' first dichotomy  yields  that
${\mathcal G \mathcal M}$ contains at least an infinite dimensional block subspace which is HI. 
\end{remark}

\section{Yardstick Vectors in ${\mathcal G \mathcal M}$}\label{S:5}
We will prove that every block basis
  in ${\mathcal G \mathcal M}$ has a further block basis whose spreading model is equivalent to the unit vector basis of $S$. Thus, we can define 
  in ${\mathcal G \mathcal M}$ the yardsticks as introduced in Section \ref{S:1}.
The following observation follows from Lemmas \ref{L:4.7} and \ref{L:4.16}, and an argument in \cite{AS2}.
 
   \begin{prop}\label{P:5.1}  Assume that $(x_n)$ is a $c$-RIS in ${\mathcal G \mathcal M}$,   for which the following condition 
   is satisfied:
   \begin{align}\label{E:5.1.1} &\text{There exists a constant $C'\ge1$, so that for all $m,k\in\N$,  $m\kleq n_1\kle n_2\kle\ldots n_m$}\\
   &\text{ in $\N$,
   all $(a_i)_{i=1}^m$ and all $z^*\in\Gamma^*_k$, it follows that}\notag\\
    &\qquad \Big|z^*\Big(\sum_{s=1}^m a_s x_{n_s}\Big)\Big|\le   \frac1{\sqrt{f(k)}} \max_{j\in J} \Big\|\sum_{s=1}^m a_s x_{n_s}\Big\|_j+ C'\max_{s\le m} |a_s|.\notag
   \end{align}
   Then the  spreading model of  $(x_n)$  is equivalent to the unit vector basis of $S$.
  More precisely there is a constant  $C$ so that for every $c$-SRIS $(x_n)$  in ${\mathcal G \mathcal M}$
  \begin{equation}\label{E:5.1.2}
   c'\Big\|\sum_{s=1}^m a_s e_s\Big\|_S\le   \Big\|\sum_{s=1}^m a_s x_{n_s}\Big\|_{{\mathcal G \mathcal M}}\le  C\Big\|\sum_{s=1}^m a_s e_s\Big\|_S,
  \end{equation}
  whenever $m\le n_1<n_2<\ldots n_{m}$ are in $\N$ and $(a_s)_{s=1}^m\subset \R$, where
$c'=c$ if $c<1$ and $c'=1-\vp_{n_1}$ if $c=1$.
  \end{prop}
  \begin{remark} Note that Lemma \ref{L:4.15} and the remark thereafter establishes a case in which the assumption  \eqref{E:5.1.1} is satisfied.
  \end{remark}
  \begin{proof}[Proof of Proposition \ref{P:5.1}]  Consider the norm $\nl\cdot\nr$ on $\coo$ given by the implicit equation 
  $$\nl x\nr= \max\Big(\| x\|_\infty,\max_{\begin{matrix}\scriptstyle l\in\N, l\ge 3 \\ \scriptstyle E_1<E_2<\ldots E_l\end{matrix}} \frac1{f(l/2)}\sum_{j=1}^l \nl E_i(x)\nr \Big),
  \text{ $x\in \coo$}$$
 and  
  recall \cite[Lemma 3.3]{AS2} which states that   $\nl\cdot\nr$  is an equivalent norm on $S$.
  
  We put 
  $$C''=C'\frac{\sqrt{f(2)}}{\sqrt{f(2)}-1}.$$
  
  By induction we will show for each $m\in\N$  and all choices of $(a_s)_{s=1}^m\subset \R$ and $m\le n_1<n_2<\ldots n_m$ in $\N$, that 
  \begin{equation}\label{E:5.1.3}
  c'\Big\|\sum_{s=1}^m a_s e_s\Big\|_S\le   \Big\|\sum_{s=1}^m a_s x_{n_s}\Big\|_{{\mathcal G \mathcal M}}\le C''  \Bnl \sum_{s=1}^m a_s e_s\Bnr\Big(1+\frac2c\sum_{s=1}^m  \vp_{n_s}\Big).
  \end{equation}  
  This will, together with the above cited result from \cite{AS2}, prove our claim. 
  The first inequality in \eqref{E:5.1.3} is clear, and it is also clear that   \eqref{E:5.1.3} holds for $m=1$. So assume that 
    \eqref{E:5.1.3}  holds for  all $m'<m$,  $m\ge 2$,  $m'\le n_1<n_2<\ldots n_{m'}$ in $\N$, and $(a_s)_{s=1}^{m'}\subset \R$. 
  Let $m\le n_1<n_2<\ldots n_m$, $(a_s)_{s=1}^m\subset \R$ and put $y=\sum_{s=1}^m a_s x_{n_s}$. 
  We distinguish between two cases: If 
  $$C''\max_{s\le m}|a_s|\ge \max_{l\in\N, l\ge 2} \|y\|_l,$$
  then we note that  for all $l\in\N$, $l\ge 2$, 
  $$\|y\|_l\le C''\max|a_s|\le C''\Bnl \sum_{s=1}^m a_s e_s \Bnr,$$
  and, thus,  for any $k\in \N$, $k\ge 2$, and $z^*\in\Gamma_k^*$, it follows from our assumption \eqref{E:5.1.1} that
  $$|z^*(y)|\le \frac1{\sqrt{f(2)}}\max_{j\in J}\|y\|_j+C'\max_{s\le m}|a_s|\le \Big[\frac{C''}{\sqrt{f(2)}} + C'\Big] \max_{s\le m}|a_s|=C'' \max_{s\le m}|a_s|.$$
 If  
   $$C''\max_{s\le m}|a_s|< \max_{l\in\N, l\ge 2} \|y\|_l,$$
  we proceed as follows.
  
   If $l\in\N$, with $f(l)\ge 2m/\vp_{n_1}$, then Lemma \ref{L:4.7} (b) implies that 
   $$\|y\|_l\le 2 \max_{i\le m} |a_s|   \le 2 \Bnl \sum_{i=1}^m a_i e_i  \Bnr  \le C'' \Bnl \sum_{i=1}^m a_i e_i  \Bnr.$$
   If  $l\ge 2$  and $f(l)\le 2m/\vp_{n_1}$, then Lemma \ref{L:4.7} (a)  yields for some choice of $0=s_0<s_1<s_2<\ldots s_{l'}$ with $l'=\min(m,2l)$, that
   \begin{align}\label{E:5.1.4}
   \|y\|_l&\le \frac1{f(l)}\Bigg[\sum_{t=1}^{l'}\Big\|\sum_{s=s_{t-1}+1}^{s_t} a_s x_{n_s}\Big\|+\frac2c\sum_{t\le l', s_t=1+s_{t-1}} \vp_{n_{s_t}} |a_{s_t}|\Bigg]\\
  &\le \frac{C''}{f(l)}\Bigg[\sum_{t=1}^{l'}  \Bnl\sum_{s=s_{t-1}+1}^{s_t} a_s e_{s}\Bnr \Big(1+\frac2c\sum_{s=1}^m\vp_{n_s} \Big)\Bigg]\notag\\
    &\text{(By applying the induction hypothesis in cases that $s_t\ge 2+s_{t-1}$)} \notag\\
    &\le  C''\Bnl \sum_{s=1}^m a_s e_s \Bnr\Big(1+\frac2c\sum_{s=1}^m\vp_{n_s} \Big).\notag
   \end{align}
Our assumption \eqref{E:5.1.1}  yields for $k\in\N$, $k\ge 2$ and $z^*\in \Gamma_k^*$, that
  $$|z^*(y)|\le \frac1{\sqrt{f(k)}} \max_{j\in\N}\|y\|_j  +C'\max_{s\le m}|a_s|\le \Bigg[\frac1{\sqrt{f(2)}}  + \frac{C'}{C''}\Bigg] \max_{j\in\N}\|y\|_j =
   \max_{j\in\N}\|y\|_j,$$
   which together with \eqref{E:5.1.4}, finishes the proof of our induction step.
    \end{proof}
 Lemma \ref{L:2.8} and Proposition \ref{P:5.1} imply  therefore
\begin{cor}\label{C:5.2} There is a constant $D$,  so that for every asymptotically isometric SRIS  $\xb=(x_n)$  in ${\mathcal G \mathcal M}$, for which $\pt(n)$ is odd for all $n\in\N$, and every
    $l\in\N$ and any $s_1<s_2<\ldots s_l$ in   $\N$
  we have
 \begin{equation}\label{E:5.2.1} 
  \frac12 \le \|y_{\xb'} (j_{s_1}, j_{s_2},\ldots j_{s_l})\|_{j_{s_i}}\le \|y_{\xb'} (j_{s_1}, j_{s_2},\ldots j_{s_l})\|\le \frac{D}{2}, \text{ for all $i=1,2,\ldots l$}
\end{equation}
where $\xb'$ is a far enough out starting tail subsequence of $\xb$.
\end{cor}

 \section{Construction of two equivalent intertwined sequences}\label{S:6}
 
 We now want to construct in a given block subspace  $Y$ of ${\mathcal G \mathcal M}$ two seminormalized block sequences $(u_n)$ and $(v_n)$,
 which are equivalent and so that $u_1<v_1<u_2<\ldots $.
   
Let $\xb=(x_i)$ be any asymptotically isometric SRIS in $Y$, so that $\pt(n)$ is odd for $n\in\N$.
Using Proposition  \ref{P:5.1} and the remark thereafter, it follows that   the spreading model of $\xb$ is equivalent
to the unit vector basis of $S$, and, since   Corollary \ref{C:5.2} applies we let $D<\infty$ be chosen so that  \eqref{E:5.2.1} holds true.

By induction we choose a  block sequence $(z_n)$  of $\xb$.  The vectors $u_n$ and $v_n$ will then be chosen
 so that $u_n<v_n$ and $z_n=u_n+v_n$.
 
 For $n=1$ we first choose $k'_1$, so that  $f(k'_1)/ k'_1<\vp_1^2$ (which means that $k'_1$ satisfies condition  \eqref{E:4.3.1}  for $n=1$), and then 
 let  
 $$z_1=\frac1D y_{\xb^{(1)}}(j_{2q_1(1)})=\frac1D \frac{f(j_{2q_1(1)})}{j_{2q_1(1)}}\sum_{t=1}^{j_{2q_1(1)}} \xb^{(1)}_t,$$
 where $q_1(1)\in\N$ is chosen large enough so that $y(j_{2q_1(1)})$ is an $\ell_1^{k_1}$-average of constant $1-\vp_1$, with $k_1\ge k'_1$
 (using Lemma \ref{L:1.13}),
 $\xb^{(1)}$ is a  tail subsequence of $\xb$, which starts   far enough  out so that $\|z_1\|\le 1$ and so that $z_1$ is  an $\ell_1^{k_1}$-average of constant
 $\frac1D$ (using Proposition \ref{P:5.1} and  Corollary \ref{C:5.2}).
 Finally we choose 
 $$u_1=\frac1D\frac{f(j_{2q_1(1)})}{j_{2q_1(1)}}\sum_{t=1}^{j_{2q_1(1)}/2} \xb^{(1)}_t \text{ and }
 v_1=\frac1D\frac{f(j_{2q_1(1)})}{j_{2q_1(1)}}\sum_{t=1+(j_{2q_1(1)}/2) }^{j_{2q_1(1)}} \xb^{(1)}_t ,$$
 (recall that the elements of $J$ are even).
 
 Assume now that for some $n\ge 2$, we have chosen $z_1<z_2< \ldots z_{n-1}$ in $B_{{\mathcal G \mathcal M}}$, and assume that  the following conditions are satisfied:
  
 -  for  each $i<n$, $z_i$ is an $\ell_1^{k_i}$-average of constant $1/D$,  so that
  \begin{equation}
  f(k_i)/k_i <\vp_i^2\text{ and } f(\vp_i\sqrt{k_i})>\frac{1}{\vp^2_i}\max\supp(z_{i-1}), \text{if $i\ge 2$}
\end{equation}\label{E:6.1}
 (in other words  $z_1<z_2<\ldots z_{n-1}$ satisfies the condition \eqref{E:4.3.1}  of the first $n-1$ elements of an RIS).

- secondly $z_i$ is of the form 
\begin{equation}\label{E:6.2}
z_i= u_i+v_i=\frac1D y_{\xb^{(i)}} \big(j_{2q_i(1)},j_{2q_i(2)},\ldots j_{2q_i(l_i)}\big),
\end{equation}
where $l_i$, and $q_i(1)<q_i(2)<\ldots <q_i(l_i)$  are in $\N$ and $\xb^{(i)}$ is a tail subsequence of $\xb$, starting far enough  out to ensure that $(z_i)_{i=1}^{n-1}$ is a block sequence
and is in $B_{{\mathcal G \mathcal M}}$ (using Corollary \ref{C:5.2}).  By Definition \ref{D:1.11}  of the yardstick vectors in Section \ref{S:1} 
we can write $z_i$ as
$z_i=\sum_{r=1}^{l_i} z(i,r)$ where
 the $z(i,r)$, $r=1,2\ldots l_i$, have pairwise disjoint support and 
so that for each $r\le l_i$ the vector
$z(i,r)$ is of the form
\begin{equation}\label{E:6.3}
z(i,r)=\frac1D\frac{f(j_{2q_i(r_t)})}{j_{2q_i(r_t)}}\sum_{s=1}^{j_{2q_i(r)}} x(i,r,s).
\end{equation}
where $x(i,r,q)$, $q=1, 2,\ldots j_{2q_i(r)}$  are elements  of the sequence $\xb^{(i)}$, and we have 
\begin{align}\label{E:6.4}
u_i&=\frac1D\sum_{r=1}^{l_i} \frac{f(j_{2q_i(r)})}{j_{2q_i(r)}}\sum_{s=1}^{j_{2q_i(r)}/2} x(i,r,s)\text{ and }\\
v_i&=\frac1D\sum_{r=1}^{l_i} \frac{f(j_{2q_i(r)})}{j_{2q_i(r)}}\sum_{s=1+(j_{2q_i(r)}/2)}^{j_{2q_i(r)}} x(i,r,s).
\end{align}

- moreover we assume that so far the following condition is satisfied:

For each sequence $\ib=(i_t: t=1,2\ldots ,l)\subset \{1,2,\ldots, n-1\}$, with $1\le i_1<i_2<\ldots i_l\le n-1$, and for each
$\rhob=(\rho_t: t=1,2\ldots ,l)\in\{ -1,1\}^l$ there is a sequence of functionals  $\zb^*(\ib,\rhob)=(z^*_t:t=1,2,\ldots l)= (z^*_{(\ib,\rhob)}(t): t=1,2\ldots, l)$ in $B_{{\mathcal G \mathcal M}^*}$ so that for
 all $t=1,2\ldots, l-1$:
\begin{align}\label{E:6.5}
&(a)\quad\text{$\textstyle\supp(z^*_t)\subset \bigcup_{s=1}^{j_{2q_i(r_t)})} \ran(x(i_t,r_t,s))$, for some $r_t\kin \{1,2\ldots l_{i_t}\}$,}\qquad\\
&(b)\quad z^*_t\kin A^*_{j_{2q_i(r_t)}}\cap{\bf Q}, \notag\\ 
&(c)\quad\textstyle z^*_t(z(i_t,r_t)) \rho_t\ge \frac1{2D}, z^*_t(u(i_t))=z^*_t(v(i_t))=\frac12 z^*_t(z(i_t,r_t)), \text{ and,}\quad\qquad\qquad \notag \\
&(d)\quad\text{if $t\kge 2$, then }2j_{2q_i(r_t)}\!=\!\sigma\big(\rho_1z^*_1,\rho_2 z^*_2,\ldots,\rho_{t-1} z^*_{t-1}\big). \notag
\end{align}

 In order to choose $z_n$ we proceed  as follows. We first choose $k'_n\in J$ so that   $f(k'_n)/k'_n<\vp_n^2$ and
  $f(\vp_n\sqrt{k'_n})\ge \frac1{\vp_n^2} \max\supp(z_{n-1})$.  
  Assume that $q_n\in\N$ satisfies the following properties:
  \begin{align}\label{E:6.6}  j_{q_n}\ge k'_n \text{ and } \sqrt{f(j_{q_n}\vp_n)}>\frac{\max\supp(z_{n-1})}{\vp_n}.
  \end{align}

  For each increasing  sequence  $\ib=(i_t: t=1,2\ldots ,l)\subset \{1,2,\ldots n-1\}$, 
  and each $\rhob=(\rho_t: t=1,2\ldots ,l)\subset\{\pm1\}$
  we can assume  that 
  \begin{equation}\label{E:6.7}
  \sigma\big( z^*_{(\ib,\rhob)}(1), z^*_{(\ib,\rhob)}(2),\ldots,z^*_{(\ib,\rhob)}(l)\big)\ge j_{q_n+1},
  \end{equation}
  for any  $\ib=(i_t: t=1,2\ldots ,l)\subset \{1,2,\ldots n-1\}$, with $1\le i_1<i_2<\ldots i_l\le n-1$, and for each
$\rhob=(\rho_t: t=1,2\ldots ,l)\subset\{\pm1\}$. Note that this can be accomplished by only perturbing the last element $z^*_{(\ib,\rhob)}(l)$,
and thus still satifying condition \eqref{E:6.5} (c) (and all the other conditions of  \eqref{E:6.5}).
 Then we consider the set
 $$\Sigma_n=\left\{\sigma\big( z^*_{(\ib,\rhob)}(1), z^*_{(\ib,\rhob)}(2),\ldots,z^*_{(\ib,\rhob)}(l)\big):\begin{matrix} 
 \ib=(i_t: t\kleq l)\subset \{1,\ldots n-1\}\text{ increasing}\\
 \rhob=(\rho_t: t=1,2\ldots ,l)\subset\{\pm1\}
 \end{matrix}\right\}$$
and order it into 
\begin{equation}\label{E:6.7a}
j_{2q_n(1)} < j_{2q_n(2)}<\ldots <j_{2q_n(l_n)}.
\end{equation}
We then choose a tail subsequence  $\xb^{(n)}$ of $\xb$  whose first element starts after $z_{n-1}$ and so that its first 
$\sum_{r=1}^{l_n}  j_{2q_n(r)}$  elements are $(1+\vp_n)C$ ($C$ as in Proposition \ref{P:5.1}) equivalent to the
first $\sum_{r=1}^{l_n}  j_{2q_n(2)}$ elements of $S$, and then put
\begin{equation}\label{E:6.7b}
z_n= \frac1D y_{\xb}(j_{2q_n(1)} ,  j_{2q_n(2)}, \ldots, j_{2q_n(l_n)}).
\end{equation}
Then $z_n>z_{n-1}$ and $\|z_n\|_{\mathcal G \mathcal M}\le 1$ by Proposition \ref{P:5.1}.
Lemma \ref{L:1.13}, Proposition \ref{P:5.1}  and Corollary \ref{C:5.2} yield that 
$z_n$ is an $\ell_1^{k_n}$ average of constant $\frac1D$, with $k_n\ge k'_n$.

 As before we can, by the definition of the yardstick vectors,  write $z_n$ as 
\begin{equation}\label{E:6.7c}
 z_n =\sum_{r=1}^{l_n} z(n,r),
 \end{equation}
where
 the  $z(n,r)$ have pairwise disjoint support and $z(n,r)$ is
 for each $r\le l_i$  
  of the form
\begin{equation*}
z(i,r)=\frac1D\frac{f(j_{2q_n(r_t)})}{j_{2q_n(r_t)}}\sum_{s=1}^{j_{2q_n(r)}} x(i,n,s),
\end{equation*}
where $x(i,n,q)$, $q=1 ,2,\ldots j_{2q_n(r)})$  are elements  of the sequence $\xb^{(n)}$, and we let 
\begin{align}\label{E:6.8}
u_n&=\frac1D\sum_{r=1}^{l_n} \frac{f(j_{2q_n(r)})}{j_{2q_n(r)}}\sum_{s=1}^{j_{2q_n(r)}/2} x(n,r,s)\text{ and }\\
v_n&=\frac1D\sum_{r=1}^{l_n} \frac{f(j_{2q_n(r)})}{j_{2q_n(r)}}\sum_{t=s+(j_{2q_n(r)}/2)}^{j_{2q_n(r)}} x(n,r,s). \notag
\end{align}
It is now easy to find for every $\ib=(i_t: t=1,2\ldots ,l)\subset \{1,2,\ldots, n\}$, with $1\le i_1<i_2<\ldots i_l\le n$, and for each
$\rhob=(\rho_t: t=1,2\ldots ,l)\in\{-1,1\}^l$  a sequence of  functionals  $\zb^*(\ib,\rhob)=(z^*_t:t=1,2,\ldots l)= (z^*_{(\ib,\rhob)}(t): t=1,2\ldots, l) \in B_{{\mathcal G \mathcal M}^*}$ so that  \eqref{E:6.5} holds. Indeed, if the last element $i_l<n$, then we already have chosen  $\zb^*(\ib,\rhob)$. So let us assume $i_l=n$. 
 Let $\ib'=(i_t:t=1,2\ldots, l-1)$ and $\rhob'=(\rho_t:  t=1,2\ldots, l-1)$.  If $\ib'$ and $\rhob'$ are empty we choose $r=1$. Otherwise we choose $r=r(\ib,\rhob)\in\{1,2\ldots,l_n\}$ so that
 $$
  j_{2q_n(r)}= \sigma\big( z^*_{(\ib',\rhob')}(1), z^*_{(\ib',\rhob')}(2),\ldots,z^*_{(\ib',\rhob')}(l-1)\big)$$
  (by choice of $q_n(i)$, $i=1,2\ldots l_n$, this is possible).
  Then choose for every $s=1,2,\ldots, j_{2q_n(r)}$, a functional $x^*_s=x^*(n,s,\ib,\rhob)\in B_{{\mathcal G \mathcal M}^*}\cap {\bf Q}$, so that 
  \begin{align*}
 &x^*_s(x(r,n,s))=x^*_{s'}(x(i,n,s')) \ge 1-2\vp_n, 
  \text{ for $s\not=s'$ in $\{1,2\ldots  j_{2q_n(r)}\}$}\\
  &\text{and } \supp(x^*_q)\subset \ran(x(r,n,q)).\end{align*}
Let 
  $$z^*_{(\ib,\rhob)}(l)=\rho_l \frac1{f(j_{2q_n(r)})} \sum_{q=1}^{j_{2q_n(r)})} x^*_q,$$
and 
    $$\zb^*_{(\ib,\rhob)}=\big(\zb^*_{(\ib',\rhob')},z^*(\ib,\rhob)(l)\big).$$
    It follows therefore that \eqref{E:6.5} is satisfied, which finishes  our recursive definition of $z_n$, $u_n$ and $v_n$.
     
    \smallskip        
    For $n\in\N$ and $l\in J\setminus \{q_n(1),q_n(2), \ldots q_n(l_n)\}$ we estimate $\|z_n\|_l$.

   Note that the construction of the $z_n$ accomplishes the following:

   If $m\in\N$ and $m\le n_1<n_2< 
   \ldots <n_m$ are in $\N$ and $(a_s)_{s=1}^m\subset \R$, we  can 
   choose $\ib=(n_s)_{s1}^m$ and $\rhob=(\text{sign}(a_s))_{s=1}^m$, and conclude that 
   $$z^*=\frac1{\sqrt{f(k)} }\sum_{t=1}^m z^*_{(\ib,\rhob)}(t)\in \Gamma^*_m,$$
 and   
 \begin{equation}\label{E:6.9}
   \Big\|\sum_{s=1}^m a_s z_{n_s}\Big\|\ge z^*\Big(\sum_{s=1}^m a_s z_{n_s} \Big)\ge \frac1{\sqrt{f(k)}}\frac1D \sum_{s=1}^m |a_s|.
   \end{equation}

    After passing to a subsequence we can assume that $(z_n)$ has a spreading model and that it is a   $\frac1D$-RIS. We  define $w_n=u_n-v_n$. Then $(w_n)$ also satisfies  \eqref{E:4.3.1} of the definition of $\frac1D$-RIS, and passing to a subsequence,
    we may also assume that $(w_n)$  has a spreading model satisfying \eqref{E:4.3.2} and  is therefore also a $\frac1D$-RIS.  
  We claim that $(w_n)$  satisfies the condition \eqref{E:5.1.1} of Proposition \ref{P:5.1}
    and it follows therefore that $(w_n)$ has a spreading model which is equivalent to the unit vector basis in $S$.

    We first estimate $\|z_n\|_l$ for $n\in\N$ and  $l\in J$.
    Let $$r_0=\max\{ r=1,2\ldots l_n: j_{2q_n(r-1)} \le l\} \text{ with $j_{2q_n(0)}:=0$}.$$
    Then by the definition of $z_n$ and the $z(n,r)$, by condition \eqref{E:2.6} on the sequence $(j_i)$, and by Lemma \ref{L:1.13},
    it follows that $\sum_{r=r_0}^{l_n} z(n,r)$ is an $\ell_1^k$ average for some $k\ge j_{2q_n(r_0)-1}>l$. It follows therefore from Lemma \ref{L:4.2} that
    \begin{equation}\label{E:6.10} 
    \Big\|\sum_{r=r_0}^{l_n} z(n,r)\Big\|_l\le \frac2{f(l)}.
    \end{equation}
    If $r=1,2\ldots r_0-1$, and thus $j_{2q_n(r)}<l$, we deduce from \eqref{E:2.9} that
    $$f(l)\ge  f( j_{2q_n(r)+1})\ge \frac{2}{\vp_{2q_n(r)+1}} j_{2q_n(r)} $$
    and thus, by Lemma \ref{L:4.7} (b)  
      \begin{equation}\label{E:6.11} 
    \| z(n,r)\|_l\le  \frac{2f(j_{2q_n(r)})}{j_{2q_n(r)}}.
        \end{equation}
    It follows therefore from \eqref{E:6.10}, \eqref{E:6.11} and \eqref{E:2.5} that for $l\in J\setminus \{q_n(1),q_n(2), \ldots q_n(l_n)\}$
    \begin{equation}\label{E:6.12} 
    \| z_n\|_l\le   \sum_{r=1}^{l_n}  \frac{2f(j_{2q_n(r)})}{j_{2q_n(r)}}  + \frac2{f(l)}\le       \frac1{f(j_{2q_n(1)-1})} + \frac2{f(l)}       \le                   \frac1{j_{q_n}}+\frac2{f(l)}.
        \end{equation}
 Using   the same  argument we observe similar inequalities for  $u_n$, $v_n$, $w_n$ :
       \begin{equation}\label{E:6.13} 
    \| u_n\|_l\le   \frac1{j_{q_n}}+\frac2{f(l)} ,\quad    \| v_n\|_l\le   \frac1{j_{q_n}}+\frac2{f(l)}  \text{ and }   \| w_n\|_l\le   \frac1{j_{q_n}}+\frac2{f(l)}.
        \end{equation}
 
      In order to verify  condition \eqref{E:5.1.1} of Proposition \ref{P:5.1}
      let $m\le n_1<n_2<\ldots n_s$ be  in $\N$  and $(a_s)_{s=1}^m\subset \R\setminus\{0\}$ we put
$y=\sum_{s=1}^m a_s w_{n_s}.$
Secondly let $k\in\N$ and $z^*\in \Gamma^*_k$.
As before we write $z^*$ as
$$z^*=\frac1{\sqrt{f(k)}}\sum_{t=1}^k z^*_t \in \Gamma^*_k,$$
with $z^*_1\in A^*_{l_1}$, and  $l_1=j_{2k'}$, for some $k'\ge k$, and $z^*_{i}\in  A^*_{l_{i}}$, with $l_i=\sigma(z^*_1,z^*_2,\ldots,z^*_{i-1})$, for $i=2,\ldots k$
  and assume that 
$$ t_0=\min\{ t=1,\ldots k:      z^*_t(y)\not=0\},$$
 exists (otherwise  $z^*(y)=0$).
 Note that the equalities in   condition \eqref{E:6.5}(c)  imply that $z^*{(\ib,\rhob)}(t)(w_j)=0$ for every $j\in\N$, every increasing sequence 
    $\ib=(i_t: t=1,2\ldots ,l)\subset \{1,2,\ldots, \}$, for each
$\rhob=(\rho_t: t=1,2\ldots ,l)\in\{-1,1\}^l$,  and for every $t=1,2\ldots l$. So it follows that
the sequence $(z^*_1,z^*_2,\ldots z^*_{t_0})$  
cannot be one of the sequences $z^*{(\ib,\rhob)}(t)$,
where    $\ib=(i_t: t=1,2\ldots ,l)\subset \{1,2,\ldots, \}$, is increasing  and
$\rhob=(\rho_t: t=1,2\ldots ,l)\in\{-1,1\}^l$. From the injectivity of $\sigma$
it follows therefore that the sets $\{l_t:t>t_0\}$ and the set 
$\{ j_{2q_n(r)}: n\in\N, r\le l_n\}$ are disjoint. 
    
     We can now  apply \eqref{E:4.13} and \eqref{E:6.13}
     deduce that 
\begin{align*}
|z^*(y)|&\le  \frac{|z^*_{t_0}(y)|}{\sqrt{f(k)}} +\frac{\max_{s\le m}|a_s|}{\sqrt{f(k)}}\left[\sum_{t\in T_0} \max_{s\in S_t} \|w_{n_s}\|_{l_t} + 1+ 2m\vp_{n_1}    \right] \notag\\
  &\qquad + \frac1{\sqrt{f(k)}} \sum_{s=s_0+1}^m |a_s| \Big|  \sum_{t\in T_s} z^*_t(w_{n_s})\Big|\notag\\
  &\le  \frac{|z^*_{t_0}(y)|}{\sqrt{f(k)}} +\frac{\max_{s\le m}|a_s|}{\sqrt{f(k)}}\Big[2+\sum_{t=t_0+1}^k \frac2{f(l_k)}  +\sum_{s=1}^{s_0} \frac1{j_{q_{n_s}}} + \sum_{s=s_0+1}^m\frac{k}{j_{q_{n_s}}}     \Big] \notag
\notag\end{align*}
where by \eqref{E:4.10}   $s_0= \min\big\{s=1,2\ldots \max \supp(w_{s-1})<\vp_{n_1} \sqrt{f(k)}\big\} $.
It follows for $s>s_0$ from  \eqref{E:6.6} that 
$$\sqrt{f(k)}< \frac{\max\supp(z_{s_0})}{\vp_{n_1}}\le \sqrt{f(\vp_{n_s}j_{q_{n_s})}}
 $$
and thus, that $k/{j_{q_{n_s}}}< \vp_{_{n_s}}\le \vp_{n_1}$ which yields,  
$$|z^*(y)|\le\frac{|z^*_{t_0}(y)|}{\sqrt{f(k)}} +5\frac{\max_{s\le m}|a_s|}{\sqrt{f(k)}} $$
and allows us to conclude from Proposition  \ref{P:5.1} that the spreading model of 
$(w_n)$ is equivalent to the unit vector basis  in $S$.

Together with \eqref{E:6.9} we     therefore  proved  the following result.
\begin{thm}\label{T:6.14}  There is a constant $c>1$  so that in  every block subspace of ${\mathcal G \mathcal M}$ we can find  block sequences $u_n$ and $v_n$,
with  $u_1<v_1<u_2<v_2<\ldots$, 
so that   $(u_n-v_n)$ has a  spreading model  which is $c$-equivalent to the unit vector basis of $S$, and  the sequences $(u_n)$, $(v_n)$ 
$(u_n+v_n)$  have spreading models which $c^{-1}$-dominate the norm $\|\cdot\|_{f^{1/2}}$ which  was introduced in Section \ref{S:1}.  I.e. if  we put $x_n=u_n, v_n$ or $u_n+v_n$, for $n\in\N$,
 and we denote by  $(E,\|\cdot\|_E)$  the spreading model of $(x_n)$  and its   basis  by $(e_j)$ then 
\begin{equation}\label{E:6.14.1}
   c\Big\| \sum_{s=1}^\infty a_s e_s\Big\|_E\ge     \| (a_s)\|_{f^{1/2}}     =     \max_{m\in\N, s_1<s_2,\ldots s_m} \frac{1}{\sqrt{f(m)}} \sum_{i=1}^m |a_{s_i}| \text{ for $(a_{s})\in c_{00}$}.
\end{equation}
\end{thm}

Thus,  Corollary \ref{C:1.10}  and Lemma \ref{L:1.9} yield our final result:

\begin{thm}\label{T:6.15} Let $(u_n)$ and $(v_n)$ be as in Theorem \ref{T:6.14}. Then there is a subsequence  
 $(n_k)$ of $\N$ so that  $(u_{n_k})$ and $(v_{n_k})$  are  equivalent.
 \end{thm}
\begin{proof} Using Corollary \ref{C:1.10} and Lemma \ref{L:1.8} twice, we may assume that
\begin{align*}
S_1&: [u_{n} :n\kin\N] \to  [u_{n}-v_{n}: n\kin\N], \quad \text {defined by } u_{n}\mapsto u_{n}-v_n, \text{ and }\\
S_2&: [v_n:n\kin\N] \to  [u_n-v_n:n\kin\N], \quad \text {defined by } v_n\mapsto u_n-v_n,\text{ for $n\kin \N$}
\end{align*}
are bounded. 
So the bounded map $Id_{|[u_n]}-S_1$ defines $u_n \mapsto v_n$, while
$Id_{|[v_n]}+S_2$ defines $v_n \mapsto u_n$, proving the claim.
\end{proof}

\begin{remark} It is worth noting that when $(u_n)$ and $(v_n)$ are intertwined and equivalent in ${\mathcal G \mathcal M}$,  the sequences $(v_n)$ and $(u_{n+1})$ are not, in general, equivalent. Otherwise the shift on $[u_n:n\kin\N]$ would be an isomorphism  and we would obtain an isomorphism of a subspace of ${\mathcal G \mathcal M}$ with its hyperplanes. But this is impossible if $(u_n)$ was picked inside  an HI subspace of ${\mathcal G \mathcal M}$.
\end{remark}
\section{Consequences of the main result}\label{S:7}

\subsection{Asymptotic unconditionality}

\

Recall that a seminormalized basis $(e_n)$ is said to be {\em asymptotically unconditional}
if there exists a constant $C \geq 1$ such that for any $k \in \N$ and any
successive blocks $k<x_1<\cdots<x_k$ on the basis, the sequence
$(x_1,\ldots,x_k)$ is $C$-unconditional. The following is an easy consequence
of Theorem \ref{T:6.14}.

\begin{prop}\label{P:6.16} The space $\mathcal G \mathcal M$ does not contain any asymptotically unconditional block sequence.
\end{prop}
We recall that the asymptotically unconditional HI space $G$ of Gowers is tight by range \cite{FR2} and therefore   contains no  intertwined and equivalent block sequences.

The sequences $(u_n)$ and $(v_n)$ are chosen in an arbitrary, but fixed subspace $Y$ of ${\mathcal G \mathcal M}$, and this is why our techniques do not seem to imply  that ${\mathcal G \mathcal M}$  is HI (although we suspect it is). This restriction is essentially technical, however, since as we shall now see, by using Gowers' Ramsey theorem, it disappears when  passing to an appropriate subspace
of ${\mathcal G \mathcal M}$.

\subsection{Applications of Gowers' Theorem}

\

Recall that Gowers' game ${\mathcal G}_X$ in a space $X$ with a basis is a game between two players, where Player 1 plays block subspaces
$Y_n$ of $X$ and Player 2 successive blocks $y_n \in Y_n$, the outcome of the game being the block-sequence $(y_n)$.

The set $b(X)$ of block-sequences of $X$ is seen as a subset of $X^\omega$ equipped with the product of the norm topology on $X$. Also for $\Delta=(\delta_n)_n$ a sequence of positive number, and $\A \subset b(X)$, the set ${\A}_{\Delta}$ is defined as
$${\A}_{\Delta}=\{(x_n) \in b(X) \del \exists (y_n) \in \A, \|y_n-x_n\| \leq \delta_n \forall n\}.$$

\begin{thm}[Gowers' Ramsey Theorem, \cite{g:dicho}]\label{ramsey} Let $X$ be a space with a basis, and $\A$ an analytic subset of $b(X)$. Let $\Delta>0$. Then there exists a block-subspace $Y$ of $X$ such that $\A \cap b(Y) = \emptyset$, or such that Player 2 has a winning strategy in Gowers' game ${\mathcal G}_Y$ to
produce an outcome in ${\A}_{\Delta}$. 
\end{thm}

Given $c \geq 1$, consider the set $\A$ of block sequences $(x_n)_n$ in ${\mathcal G \mathcal M}$
such that $(x_{2n}-x_{2n+1})$ has a spreading model which is $c$-equivalent to the unit vector basis of $S$, and  the sequences $(x_{2n})$, $(x_{2n+1})$ 
$(x_{2n}+x_{2n+1})$  have a spreading model which $c^{-1}$ dominate the norm $\|\cdot\|_{f^{1/2}}$. It is easily checked that $\A$ is Borel. So up to modifying the constant $c$ to take into account a small enough perturbation $\Delta$, we may apply Gowers' Theorem to find a block-subspace $Y$ of
${\mathcal G \mathcal M}$, so that the vectors $u_n$ and $v_n$ of Theorem  \ref{T:6.14} may be chosen in arbitrary block-subspaces of $Y$ prescribed by Player~1. 

\begin{prop}\label{HIHI} There exists $c \geq 1$ and a block subspace of ${\mathcal G \mathcal M}$ in which Player 2 has a winning strategy to produce
 $u_1<v_1<u_2<v_2<\ldots$, 
so that   $(u_n-v_n)$ has a  spreading model  which is $c$-equivalent to the unit vector basis of $S$, and  the sequences $(u_n)$, $(v_n)$ 
$(u_n+v_n)$  have a spreading models which  $c^{-1}$ dominate the norm $\|\cdot\|_{f^{1/2}}$.
\end{prop}

Note that $\|\sum_{i=1}^m e_i\|_{f^{1/2}}=mf(m)^{-1/2}$ while
$\|\sum_{i=1}^m e_i\|_S=mf(m)^{-1}$. So for any $\epsilon>0$, one can find $m \in \N$ with following property: for any $U, V$ block subspaces of $Y$, there exist $u_1<v_1<\cdots<u_m<v_m$, with $u_i \in U, v_i \in V$ for each $i$, such that
$\|\sum_{i=1}^m u_i-v_i\|<\epsilon \|\sum_{i=1}^m u_i+v_i\|$,
which of course implies the  HI property. This property is actually the uniform version of the HI property which appears as the counterpart of asymptotic unconditionality in the dichotomy proved by Wagner  \cite{W}.

\

The third dichotomy implies that we may assume that the space of Proposition \ref{HIHI} is tight, and the fourth dichotomy that it is subsequentially minimal. Actually slightly more may be observed.  

\begin{thm}\label{goma} There exists a tight HI block-subspace ${\mathcal X}_{GM}$ of ${\mathcal G \mathcal M}$ with a normalized basis which is subsequentially minimal.
 More precisely, there exists $c \geq 1$, such that for any block-subspace $Y$ of ${\mathcal X}_{GM}$, there exists a block-sequence
$(y_k)$ of $Y$ and a subsequence $(f_k)$ of the basis  of ${\mathcal X}_{GM}$ such that
\begin{itemize}
\item[(a)] $y_1<f_{1}<y_2<f_{2}<\cdots$ 
\item[(b)] $(y_k), (f_k), (y_k+f_k)$ have spreading models which $c^{-1}$ dominate the norm $\|\cdot\|_{f^{1/2}}$,
\item[(c)] $(y_k-f_k)$ has  a spreading model which is $c$-equivalent to the unit vector basis of $S$,
\item[(d)] consequently, 
$(f_k)$ is equivalent to $(y_{k})$.
\end{itemize}
\end{thm}

This is a variation on \cite[Proposition 6.5]{FR}. Since the proof is much shorter than the demonstration of the fourth dichotomy, we give a sketch of it.

\begin{proof} Let $\A \subset b({\mathcal G \mathcal M})$ be defined as after Theorem \ref{ramsey}. Using Gowers'  first dichotomy  (see Theorem \ref{T:0.1}), the fact that no
 HI space has a minimal subspace, and  the third dichotomy proven in \cite{FR} (see Theorem \ref{T:0.3})
 we may pass to an HI tight subspace.
  By Theorem \ref{T:6.15} and Gowers' Ramsey Theorem (Theorem \ref{ramsey})
  we can, after  modifying $c$,   assume that Player 2 has a winning strategy in Gowers'
game to play inside  $\A$; also we may and shall only use blocks with rational coordinates in this proof (and assume Gowers' game is played with such blocks). Then the finite block-sequences of initial moves prescribed by the winning strategy of Player~2 form a non-empty tree $T$ which does not have any maximal elements.
  We denote by $[T]$ the infinite block sequences $(x_j)$ for which all the initial segments $(x_j)_{j=1}^n$, $n\in\N$, lie in $T$. 
   Then $[T]\subseteq \A$ and for all
$(y_0,\ldots,y_m)\in T$ and all block sequences $(z_n)$, there is a block
$y_{m+1}$ of $(z_n)$ such that $(y_0,\ldots,y_m, y_{m+1})\in T$, \cite[Lemma 6.4]{FR}.
Since $T$ is
countable, we can construct inductively a block sequence $(v_n)$ 
such that for all $(u_0,\ldots,u_m)\in T$ there is some $v_n$ with
$(u_0,\ldots,u_m,v_n)\in T$.

We claim that ${\mathcal X}_{GM}:=[v_n, n \in \N]$ works. 
Indeed if $(z_n)$ is any block sequence of $(v_n)$,
we may construct inductively a block-sequence $(y_i)$ of $(z_n)$ and a subsequence $(f_i)$  of $(v_n)$ such that
$(y_0,f_0,\cdots,y_n,f_n)$ belongs to $T$ for all $n$. Therefore 
$(y_0,f_0,y_1,f_1,\cdots)$ belongs to $\A$.
Finally the normalized basis $(v_n/\|v_n\|)$ of ${\mathcal X}_{GM}$ satisfies the conclusion
of Theorem \ref{goma}.
\end{proof}

Since this construction can be done in any block-subspace of ${\mathcal G \mathcal M}$,  we may assume that ${\mathcal X}_{GM}$ is actually sequentially minimal.

\subsection{Local minimality}

\

We briefly expose the fifth dichotomy obtained in \cite{FR},
which is related to the second general kind of tightness called  tightness with constants. A space $X=[e_n]$ is
tight {\em with constants} when for for every infinite dimensional space $Y$, the
sequence of successive subsets $I_0<I_1<\ldots$ of $\N$ witnessing the
tightness of $Y$ in $X$ may be chosen so that
$Y \not\sqsubseteq_K [e_n \del n \in \N \setminus I_K]$ for each $K$. Equivalent no infinite dimensional space embeds uniformly into the tail subspaces of $X$ \cite[Proposition 4.1]{FR}. This is the case for Tsirelson's space $T$ or its $p$-convexified version $T^{(p)}$.

On the other hand we already mentioned that a space $X$ is said to be {\em locally minimal} if there exists a constant $K \geq 1$ such that every finite dimensional subspace of $X$ $K$-embeds into every infinite dimensional subspace of $X$. 

\begin{thm}[Fifth dichotomy \cite{FR}] Any Banach space  contains a subspace with a basis which is either tight
with constants or  locally minimal. \end{thm}

Since $S$ contains $\ell_\infty^n$'s uniformly and since ${\mathcal G \mathcal M}$ is saturated with sequences with  spreading model $c$-equivalent to the basis of $S$, ${\mathcal G \mathcal M}$ also contains $\ell_\infty^n$'s uniformly in every subspace. So by the universal properties of these spaces, ${\mathcal G \mathcal M}$ is locally minimal.

\begin{thm} There exists a locally and sequentially minimal HI Banach space. 
\end{thm}

Since an HI space does not contain a minimal subspace, this answers  \cite[Problem 5.2]{FR2}, that is, the space ${\mathcal G \mathcal M}$ demonstrates that there are other forms of tightness than tightness by range or with constants.

The fifth dichotomy and a dichotomy due to A. Tcaciuc \cite{Tc} are used in \cite{FR} to refine the types (1)-(6) into  subclasses. In their terminology, ${\mathcal X}_{GM}$ is of type (2b).

\subsection{Open problems}

\

The most important problem which remains open in Gowers' classification program is whether there exist spaces of type (4).
Note that such a space would satisfy the criterion of Casazza, and therefore would not be isomorphic to its proper subspaces.

\begin{prob} Find a space with an unconditional basis, tight by range and quasi-minimal.
\end{prob}

The nature of the tightness of ${\mathcal X}_{GM}$ remains to be understood. This property is a consequence of the non-minimality of HI spaces and of the third dichotomy, with no information on how the sequence $(I_n)$ of subsets of $\N$ depends on the subspace $Y$.

\begin{prob} Find information on the sequences $(I_n)$ in the definition of the tightness of ${\mathcal X}_{GM}$. Is ${\mathcal G \mathcal M}$ or $GM$ itself tight? \end{prob}

C. Rosendal \cite{R} defined notions of $\alpha$-minimality and $\alpha$-tightness, where $\alpha<\omega_1$ is an ordinal. Local minimality implies that ${\mathcal X}_{GM}$ is $\omega^2$-minimal and not $\omega$-tight. On the other hand, being tight, it must be $\alpha$-tight for some $\alpha<\omega_1$, \cite[Theorem 3]{R}.

\begin{prob} Find $\min \{\alpha \in \omega_1 \del {\mathcal X}_{GM} {\rm \ is\ } \alpha{\rm-tight}\}$. \end{prob}

It is unknown whether an HI space may be tight with constants. With the exception of
the  uniformly convex HI space of \cite{F:unif},
examples of the Gowers-Maurey family usually
contain $\ell_\infty^n$'s uniformly - and therefore are locally minimal. 

\begin{prob} Find an HI space which is tight with constants.
\end{prob}

  \begin{flushleft}

{\em Address of V. Ferenczi:}\\

Departamento de Matem\'atica,\\

Instituto de Matem\'atica e Estat\' \i stica,\\

Universidade de S\~ao Paulo,\\

rua do Mat\~ao, 1010, \\

05508-090 S\~ao Paulo, SP,\\

Brazil.\\
\texttt{ferenczi@ime.usp.br}
\end{flushleft}

\

\begin{flushleft}
{\em Address of Th. Schlumprecht:}\\

Department of Mathematics,\\

Texas A\&M University,\\

College Station, Texas, 77843,

USA  \\
\texttt{schlump@math.tamu.edu}

\end{flushleft}

  \end{document}